\documentclass[reqno]{amsart}
\usepackage{amsfonts}
\usepackage[left=3.30cm,right=3.30cm]{geometry}
\usepackage{amsmath}

\newtheorem{theorem}{Theorem}
\theoremstyle{plain}
\newtheorem{acknowledgement}{Acknowledgement}

\newtheorem{corollary}{Corollary}

\newtheorem{definition}{Definition}
\newtheorem{example}{Example}

\newtheorem{lemma}{Lemma}

\newtheorem{proposition}{Proposition}
\newtheorem{remark}{Remark}

\DeclareMathOperator{\Div}{div}
 \numberwithin{equation}{section}

\begin{document}
\title[Homogenization of nonlinear SPDEs]{Homogenization of nonlinear
stochastic partial differential equations in a general ergodic environment}
\author{Paul Andr\'{e} Razafimandimby}
\address{Department of Mathematics and Applied Mathematics, University of
Pretoria, Pretoria 0002, South Africa (P.A. Razafimandimby)}
\email{paulrazafi@gmail.com}

\author{Jean Louis Woukeng}
\address{Department of Mathematics and Computer Science, University of
Dschang, P.O. Box 67, Dschang, Cameroon (J.L. Woukeng)}
\email{jwoukeng@yahoo.fr}
\date{}

\begin{abstract}
In this paper, we show that the concept of sigma-convergence associated to
stochastic processes can tackle the homogenization of stochastic partial
differential equations. In this regard, the homogenization of a stochastic
nonlinear partial differential equation is addressed. Using some deep
compactness results such as the Prokhorov and Skorokhod theorems, we prove
that the sequence of solutions of this problem converges in probability
towards the solution of an equation of the same type. To proceed with, we
use the concept of\textit{\ sigma-convergence for stochastic processes},
which takes into account both the deterministic and random behaviours of the
solutions of the problem.
\end{abstract}

\maketitle

\section{Introduction}

Algebras with mean value have been highly efficient in deterministic
homogenization theory. It is now a well known fact that given a partial
differential equation (PDE) with oscillating coefficients, one can always,
under some structural constraints on its coefficients, solve some
homogenization problems related to this PDE.

Contrasted with deterministic homogenization, very few results are available
as regards the homogenization of stochastic PDEs (SPDEs). We may cite \cite%
{Bensoussan1, Ichihara1, Ichihara2, Sango, Wang1, Wang2} in that context. In
the just mentioned previous work, the homogenization of SPDEs is studied
under the periodicity assumption on the coefficients of the equations
considered. In addition, the convergence method used is either the
G-convergence method \cite{Bensoussan1, Ichihara1, Ichihara2} or the
two-scale convergence method \cite{Wang1, Wang2}. Given the nature both
random and deterministic of the solutions of these equations, it is more
convenient to use an appropriate method taking into account both these two
types of behaviour. As regards the SPDEs in a general ergodic environment,
no results is available so far. The first attempt to generalize this to
SPDEs beyond the periodic context is undertaken in \cite{WoukengArxiv} in
which the authors consider the homogenization problem for a SPDE in an
almost periodic setting. The present work is therefore the first one in
which such a problem is considered.

To be more precise, we are concerned with the homogenization problem for the
following nonlinear stochastic partial differential equation
\begin{equation}
\left\{
\begin{array}{l}
du_{\varepsilon }=\left( \Div a\left( x,t,\frac{x}{\varepsilon },\frac{t}{%
\varepsilon },u_{\varepsilon },Du_{\varepsilon }\right) -a_{0}\left( x,t,%
\frac{x}{\varepsilon },\frac{t}{\varepsilon },u_{\varepsilon }\right)
\right) dt+M\left( \frac{x}{\varepsilon },\frac{t}{\varepsilon }%
,u_{\varepsilon }\right) dW\text{\ \ in }Q_{T} \\
u_{\varepsilon }=0\text{\ \ on }\partial Q\times (0,T) \\
u_{\varepsilon }(x,0)=u^{0}(x)\text{\ \ in }Q,%
\end{array}%
\right.  \label{1.1}
\end{equation}%
where $Q_{T}=Q\times (0,T)$, $Q$ being a Lipschitz domain in $\mathbb{R}^{N}$
with smooth boundary $\partial Q$, $T$ is a positive real number and $W$ is
a cylindrical standard Wiener process defined on a given probability space $%
(\Omega ,\mathcal{F},\mathbb{P})$. Under a suitable assumption on the
coefficients of (\ref{1.1}) we prove that the sequence of solutions to (\ref%
{1.1}) converges to the solution of an equation of the same type as (\ref%
{1.1}). In view of the result obtained, one might be tempted to believe that
the homogenization process of an SPDE is summarized in the homogenization of
its deterministic part, added to the average of its stochastic part. This is
not true in general. Indeed, one can obtain, passing to the limit, a
homogenized equation of a type completely different from that of the initial
problem; see e.g., \cite{Wang2}.

The paper is presented as follows. In Section 2, we give some fundamentals
of generalized Besicovitch spaces. Section 3 deals with the concept of
sigma-convergence for stochastic processes. We state therein some
compactness results that will be used in the sequel. In Section 4, we state
the problem and prove some fundamental estimates. In Section 5 we collect
some useful results necessary to the homogenization part, and we use them in
Section 6 to study the homogenization of (\ref{1.1}). We prove there the
global homogenization result and we derive the homogenized problem. Section
7 deals with a corrector-type result. Finally in Section 8, we apply the
result of Section 6 to some concrete physical situations.

Unless otherwise specified, vector spaces throughout are assumed to be real
vector spaces, and scalar functions are assumed to take real values. We
shall always assume that the numerical space $\mathbb{R}^{m}$ (integer $%
m\geq 1$) and its open sets are each equipped with the Lebesgue measure $%
dx=dx_{1}...dx_{m}$.

\section{Some properties of the generalized Besicovitch spaces}

We begin this section by recalling some important properties of algebras
with mean value \cite{Jikov, Casado, NA, Zhikov4}. By an algebra with mean
value (algebra wmv, in short) on $\mathbb{R}^{N}$ we mean any closed
subalgebra $A$ of the $\mathcal{C}$*-algebra of bounded uniformly continuous
functions $BUC(\mathbb{R}^{N})$ which contains the constants, is translation
invariant ($u(\cdot +a)\in A$ for any $u\in A$ and each $a\in \mathbb{R}^{N}$%
) and is such that each element possesses a mean value in the following
sense:

\begin{itemize}
\item[(\textit{MV})] For each $u\in A$, the sequence $(u^{\varepsilon
})_{\varepsilon >0}$ (where $u^{\varepsilon }(x)=u(x/\varepsilon )$, $x\in
\mathbb{R}^{N}$) weakly $\ast $-converges in $L^{\infty }(\mathbb{R}^{N})$
to some constant real-valued function $M(u)$ as $\varepsilon \rightarrow 0$.
\end{itemize}

It is known that $A$ (endowed with the sup norm topology) is a commutative $%
\mathcal{C}$*-algebra with identity. We denote by $\Delta (A)$ the spectrum
of $A$ and by $\mathcal{G}$ the Gelfand transformation on $A$. We recall
that $\Delta (A)$ (a subset of the topological dual $A^{\prime }$ of $A$) is
the set of all nonzero multiplicative linear functionals on $A$, and $%
\mathcal{G}$ is the mapping of $A$ into $\mathcal{C}(\Delta (A))$ such that $%
\mathcal{G}(u)(s)=\left\langle s,u\right\rangle $ ($s\in \Delta (A)$), where
$\left\langle ,\right\rangle $ denotes the duality pairing between $%
A^{\prime }$ and $A$. We endow $\Delta (A)$ with the relative weak$\ast $
topology on $A^{\prime }$. Then using the well-known theorem of Stone (see
e.g., \cite[Theorem IV.6.18, p. 274]{DS}) one can easily show that the
spectrum $\Delta (A)$ is a compact topological space, and the Gelfand
transformation $\mathcal{G}$ is an isometric isomorphism identifying $A$
with $\mathcal{C}(\Delta (A))$ (the continuous functions on $\Delta (A)$) as
$\mathcal{C}$*-algebras. Next, since each element of $A$ possesses a mean
value, this yields an application $u\mapsto M(u)$ (denoted by $M$ and called
the mean value) which is a nonnegative continuous linear functional on $A$
with $M(1)=1$, and so provides us with a linear nonnegative functional $\psi
\mapsto M_{1}(\psi )=M(\mathcal{G}^{-1}(\psi ))$ defined on $\mathcal{C}%
(\Delta (A))=\mathcal{G}(A)$, which is clearly bounded. Therefore, by the
Riesz-Markov theorem, $M_{1}(\psi )$ is representable by integration with
respect to some Radon measure $\beta $ (of total mass $1$) in $\Delta (A)$,
called the $M$\textit{-measure} for $A$ \cite{Hom1}. It is a fact that we
have
\begin{equation*}
M(u)=\int_{\Delta (A)}\mathcal{G}(u)d\beta \text{\ for }u\in A\text{.}
\end{equation*}

Next, to any algebra with mean value $A$ are associated the following
subspaces: $A^{m}=\{\psi \in \mathcal{C}^{m}(\mathbb{R}^{N}):$ $%
D_{y}^{\alpha }\psi \in A$ for every $\alpha =(\alpha _{1},...,\alpha
_{N})\in \mathbb{N}^{N}$ with $\left\vert \alpha \right\vert \leq m\}$
(where $D_{y}^{\alpha }\psi =\partial ^{\left\vert \alpha \right\vert }\psi
/\partial y_{1}^{\alpha _{1}}\cdot \cdot \cdot \partial y_{N}^{\alpha _{N}}$
and integer $m\geq 1$). Endowed with the norm $\left\Vert \left\vert
u\right\vert \right\Vert _{m}=\sup_{\left\vert \alpha \right\vert \leq
m}\left\Vert D_{y}^{\alpha }\psi \right\Vert _{\infty }$, $A^{m}$ is a
Banach space. We also define the space $A^{\infty }$ as the space of $\psi
\in \mathcal{C}^{\infty }(\mathbb{R}_{y}^{N})$ such that $D_{y}^{\alpha
}\psi =\frac{\partial ^{\left\vert \alpha \right\vert }\psi }{\partial
y_{1}^{\alpha _{1}}\cdot \cdot \cdot \partial y_{N}^{\alpha _{N}}}\in A$ for
every $\alpha =(\alpha _{1},...,\alpha _{N})\in \mathbb{N}^{N}$. Endowed
with a suitable locally convex topology defined by the family of norms $%
\left\Vert \left\vert \cdot \right\vert \right\Vert _{m}$, $A^{\infty }$ is
a Fr\'{e}chet space.

Now, the partial derivative of index $i$ ($1\leq i\leq N$) on $\Delta (A)$
is defined to be the mapping $\partial _{i}=\mathcal{G}\circ \partial
/\partial y_{i}\circ \mathcal{G}^{-1}$ (usual composition) of $\mathcal{D}%
^{1}(\Delta (A))=\{\varphi \in \mathcal{C}(\Delta (A)):\mathcal{G}%
^{-1}(\varphi )\in A^{1}\}$ into $\mathcal{C}(\Delta (A))$. Higher order
derivatives are defined analogously. At the present time, let $\mathcal{D}%
(\Delta (A))=\{\varphi \in \mathcal{C}(\Delta (A)):$ $\mathcal{G}%
^{-1}(\varphi )\in A^{\infty }\}.$ Endowed with a suitable locally convex
topology $\mathcal{D}(\Delta (A))$) is a Fr\'{e}chet space and further, $%
\mathcal{G}$ viewed as defined on $A^{\infty }$ is a topological isomorphism
of $A^{\infty }$ onto $\mathcal{D}(\Delta (A))$.

Analogously to the space $\mathcal{D}^{\prime }(\mathbb{R}^{N})$, we now
define the space of distributions on $\Delta (A)$ to be the space of all
continuous linear form on $\mathcal{D}(\Delta (A))$. We denote it by $%
\mathcal{D}^{\prime }(\Delta (A))$ and we endow it with the strong dual
topology. Since $A^{\infty }$ is dense in $A$ (see \cite[Proposition 2.3]%
{ACAP}), it is easy to see that the space $L^{p}(\Delta (A))$ ($1\leq p\leq
\infty $) is a subspace of $\mathcal{D}^{\prime }(\Delta (A))$ (with
continuous embedding), so that one may define the Sobolev spaces on $\Delta
(A)$ as follows.
\begin{equation*}
W^{1,p}(\Delta (A))=\{u\in L^{p}(\Delta (A)):\text{ }\partial _{i}u\in
L^{p}(\Delta (A))\text{ (}1\leq i\leq d\text{)}\}\;(1\leq p<\infty )
\end{equation*}%
where the derivative $\partial _{i}u$ is taken in the distribution sense on $%
\Delta (A)$. We equip $W^{1,p}(\Delta (A))$ with the norm

\begin{equation*}
\begin{array}{l}
||u||_{W^{1,p}(\Delta (A))}=\left[ ||u||_{L^{p}(\Delta
(A))}^{p}+\sum_{i=1}^{N}||\partial _{i}u||_{L^{p}(\Delta (A))}^{p}\right] ^{%
\frac{1}{p}}\text{ \thinspace }\left( u\in W^{1,p}(\Delta (A))\right) , \\
1\leq p<\infty ,%
\end{array}%
\end{equation*}%
which makes it a Banach space. To the above space is attached the space
\begin{equation*}
W^{1,p}(\Delta (A))/\mathbb{R}=\{u\in W^{1,p}(\Delta (A)):\int_{\Delta
(A)}ud\beta =0\}
\end{equation*}%
equipped with the seminorm $u\mapsto (\sum_{i=1}^{N}||\partial
_{i}u||_{L^{p}(\Delta (A))}^{p})^{1/p}$, and its separated completion $%
W_{\#}^{1,p}(\Delta (A))$. We will see in the sequel that $%
W_{\#}^{1,p}(\Delta (A))$ is in fact the completion of $W^{1,p}(\Delta (A))/%
\mathbb{R}$ since $A$ will be taken to be an \textit{ergodic} algebra; see
the last part of this section.

The concept of a product algebra wmv will be useful in our study. Let $A_{y}$
(resp. $A_{\tau }$) be an algebra wmv on $\mathbb{R}_{y}^{N}$ (resp. $%
\mathbb{R}_{\tau }$). We define the product algebra wmv $A_{y}\odot A_{\tau
} $ as the closure in $BUC(\mathbb{R}^{N+1})$ of the tensor product $%
A_{y}\otimes A_{\tau }=\{\sum_{\text{finite}}u_{i}\otimes v_{i}:u_{i}\in
A_{y}$ and $v_{i}\in A_{\tau }\}$. This defines an algebra wmv on $\mathbb{R}%
^{N+1}$. A characterization of these products is given in the following
result whose proof can be found in \cite{CMP}.

\begin{theorem}
\label{t2.1}Let $A_{y}$, $A_{\tau }$ and $A$ be as above. For $f\in BUC(%
\mathbb{R}_{y,\tau }^{N+1})$, we define $f_{y}\in BUC(\mathbb{R}_{\tau })$
and $f^{\tau }\in BUC(\mathbb{R}_{y}^{N})$ by
\begin{equation*}
f_{y}(\tau )=f^{\tau }(y)=f(y,\tau )\text{\ for }(y,\tau )\in \mathbb{R}%
_{y}^{N}\times \mathbb{R}_{\tau }
\end{equation*}%
and put
\begin{equation*}
B_{f}=\{f^{\tau }:\tau \in \mathbb{R}\},\;C_{f}=\{f_{y}:y\in \mathbb{R}%
^{N}\}.\;\;\;\;\;\;\;
\end{equation*}%
Then $B_{f}\subset A_{y}$ and $C_{f}\subset A_{\tau }$ for every $f\in A$.
Also for $f\in A$ both $B_{f}$ and $C_{f}$ are relatively compact in $A_{y}$
and in $A_{\tau }$ respectively (in the sup norm topology).
\end{theorem}

Let $AP(\mathbb{R}^{N})$ denote the space of all Bohr almost periodic
functions on $\mathbb{R}^{N}$ \cite{Besicovitch, Bohr}, that is the algebra
of functions in $\mathcal{B}(\mathbb{R}^{N})$ that are uniformly
approximated by finite linear combinations of functions in the set $%
\{y\mapsto \cos (k\cdot y),y\mapsto \sin (k\cdot y):k\in \mathbb{R}^{N}\}$.
It is well-known that $AP(\mathbb{R}^{N})$ is an algebra wmv on $\mathbb{R}%
^{N}$. As an example we have $AP(\mathbb{R}_{y}^{N})\odot AP(\mathbb{R}%
_{\tau })=AP(\mathbb{R}_{y}^{N}\times \mathbb{R}_{\tau })$. We also have
that $\mathcal{C}_{\text{per}}(Y)\odot \mathcal{C}_{\text{per}}(Z)=\mathcal{C%
}_{\text{per}}(Y\times Z)$ where $Y=\left( 0,1\right) ^{N}$ and $Z=\left(
0,1\right) $. This follows from the identification $\mathcal{C}_{\text{per}%
}(Y)=\mathcal{C}(\mathbb{T}^{N})$ where $\mathbb{T}^{N}$ is the $N$-torus in
$\mathbb{R}^{N}$. Similarly we have $\mathcal{C}_{\text{per}}(Z)\odot AP(%
\mathbb{R}_{y}^{N})=\mathcal{C}_{\text{per}}(Z;AP(\mathbb{R}_{y}^{N}))$.
Other examples of product algebras wmv can be given.\bigskip

Next, let $B_{A}^{p}$ ($1\leq p<\infty $) denote the Besicovitch space
associated to $A$, that is the closure of $A$ with respect to the
Besicovitch seminorm
\begin{equation*}
\left\Vert u\right\Vert _{p}=\left( \underset{r\rightarrow +\infty }{\lim
\sup }\frac{1}{\left\vert B_{r}\right\vert }\int_{B_{r}}\left\vert
u(y)\right\vert ^{p}dy\right) ^{1/p}
\end{equation*}%
where $B_{r}$ is the open ball of $\mathbb{R}^{N}$ of radius $r$. It is
known that $B_{A}^{p}$ is a complete seminormed vector space verifying $%
B_{A}^{q}\subset B_{A}^{p}$ for $1\leq p\leq q<\infty $. From this last
property one may naturally define the space $B_{A}^{\infty }$ as follows:
\begin{equation*}
B_{A}^{\infty }=\{f\in \cap _{1\leq p<\infty }B_{A}^{p}:\sup_{1\leq p<\infty
}\left\Vert f\right\Vert _{p}<\infty \}\text{.}\;\;\;\;\;\;\;\;\;
\end{equation*}%
We endow $B_{A}^{\infty }$ with the seminorm $\left[ f\right] _{\infty
}=\sup_{1\leq p<\infty }\left\Vert f\right\Vert _{p}$, which makes it a
complete seminormed space. We recall that the spaces $B_{A}^{p}$ ($1\leq
p\leq \infty $) are not in general Fr\'{e}chet spaces since they are not
separated in general. The following properties are worth noticing \cite{CMP,
NA}:

\begin{itemize}
\item[(\textbf{1)}] The Gelfand transformation $\mathcal{G}:A\rightarrow
\mathcal{C}(\Delta (A))$ extends by continuity to a unique continuous linear
mapping, still denoted by $\mathcal{G}$, of $B_{A}^{p}$ into $L^{p}(\Delta
(A))$, which in turn induces an isometric isomorphism $\mathcal{G}_{1}$, of $%
B_{A}^{p}/\mathcal{N}=\mathcal{B}_{A}^{p}$ onto $L^{p}(\Delta (A))$ (where $%
\mathcal{N}=\{u\in B_{A}^{p}:\mathcal{G}(u)=0\}$). Furthermore if $u\in
B_{A}^{p}\cap L^{\infty }(\mathbb{R}^{N})$ then $\mathcal{G}(u)\in L^{\infty
}(\Delta (A))$ and $\left\Vert \mathcal{G}(u)\right\Vert _{L^{\infty
}(\Delta (A))}\leq \left\Vert u\right\Vert _{L^{\infty }(\mathbb{R}^{N})}$.

\item[(\textbf{2)}] The mean value $M$ viewed as defined on $A$, extends by
continuity to a positive continuous linear form (still denoted by $M$) on $%
B_{A}^{p}$ satisfying $M(u)=\int_{\Delta (A)}\mathcal{G}(u)d\beta $ ($u\in
B_{A}^{p}$). Furthermore, $M(\tau _{a}u)=M(u)$ for each $u\in B_{A}^{p}$ and
all $a\in \mathbb{R}^{N}$, where $\tau _{a}u(z)=u(z+a)$ for almost all $z\in
\mathbb{R}^{N}$. Moreover for $u\in B_{A}^{p}$ we have $\left\Vert
u\right\Vert _{p}=\left[ M(\left\vert u\right\vert ^{p})\right] ^{1/p}$.
\end{itemize}

In this work, we will deal with ergodic algebras (see \cite{Jikov, Zhikov4}%
). Let us recall the definition of an ergodic algebra.

\begin{definition}
\label{d2.1}\emph{An algebra wmv }$A$ \emph{on }$\mathbb{R}^{N}$\emph{\ is }%
ergodic\emph{\ if for every }$u\in B_{A}^{1}$\emph{\ such that }$\left\Vert
u-u(\cdot +a)\right\Vert _{1}=0$\emph{\ for every }$a\in \mathbb{R}^{N}$%
\emph{\ we have }$\left\Vert u-M(u)\right\Vert _{1}=0$\emph{.}
\end{definition}

The class of ergodic algebra plays a crucial role in homogenization theory
as it will be seen in the following sections.

In order to simplify the text, we will henceforth use the same letter $u$
(if there is no danger of confusion) to denote the equivalence class of an
element $u\in B_{A}^{p}$. The symbol $\varrho $ will denote the canonical
mapping of $B_{A}^{p}$ onto $\mathcal{B}_{A}^{p}=B_{A}^{p}/\mathcal{N}$. Our
goal here is to define another space attached to $\mathcal{B}_{A}^{p}$. For
that purpose, let us recall that the partial derivative of index $1\leq
i\leq N$ of a distribution $u\in \mathcal{D}^{\prime }(\Delta (A))$, denoted
by $\partial _{i}u$, is defined as follows:

\begin{equation*}
\left\langle \partial _{i}u,\varphi \right\rangle =-\left\langle u,\partial
_{i}\varphi \right\rangle \text{\ for any }\varphi \in \mathcal{D}(\Delta
(A)).\;
\end{equation*}%
With this in mind, we define the formal derivative of index $i$, denoted by $%
\overline{\partial }/\partial y_{i}$, as follows:
\begin{equation*}
\frac{\overline{\partial }}{\partial y_{i}}=\mathcal{G}_{1}^{-1}\circ
\partial _{i}\circ \mathcal{G}_{1}.
\end{equation*}%
Considered as defined from $\mathcal{B}_{A}^{p}$ into itself, it is an
unbounded operator with domain $\mathcal{D}_{i}=\{u\in \mathcal{B}_{A}^{p}:%
\overline{\partial }u/\partial y_{i}\in \mathcal{B}_{A}^{p}\}$. We set
\begin{equation*}
\mathcal{B}_{A}^{1,p}=\cap _{1\leq i\leq N}\mathcal{D}_{i}\equiv \left\{
u\in \mathcal{B}_{A}^{p}:\frac{\overline{\partial }u}{\partial y_{i}}\in
\mathcal{B}_{A}^{p}\;\text{for }1\leq i\leq N\right\} .
\end{equation*}%
Since $\overline{\partial }/\partial y_{i}$ is closed, $\mathcal{B}%
_{A}^{1,p} $ is a Banach space under the norm
\begin{equation*}
\left\| u\right\| _{\mathcal{B}_{A}^{1,p}}=\left[ \left\| u\right\|
_{p}^{p}+\sum_{i=1}^{N}\left\| \frac{\overline{\partial }u}{\partial y_{i}}%
\right\| _{p}^{p}\right] ^{1/p}\text{\ \ }(u\in \mathcal{B}_{A}^{1,p}).
\end{equation*}%
Moreover, the restriction of $\mathcal{G}_{1}$ to $\mathcal{B}_{A}^{1,p}$ is
an isometric isomorphism of $\mathcal{B}_{A}^{1,p}$ onto $W^{1,p}(\Delta
(A)) $. We assume for the remainder of this section that $A$ is ergodic.
Then according to Definition \ref{d2.1}, the only elements of $\mathcal{B}%
_{A}^{1}$ that are $\left\| \cdot \right\| _{1}$-invariant are constant
functions. This infers that if $u\in \mathcal{B}_{A}^{1}$ satisfies $%
\overline{D}_{y}u=0 $, then $u$ is constant. Indeed it can can be shown that
$u\in \mathcal{B}_{A}^{1}$ is $\left\| \cdot \right\| _{1}$-invariant if and
only if $\overline{D}_{y}u=0$. So the mapping
\begin{equation*}
\left\| \overline{D}_{y}\cdot \right\| _{p}:u\mapsto \left\| \overline{D}%
_{y}u\right\| _{p}:=\left( \sum_{i=1}^{N}\left\| \frac{\overline{\partial }u%
}{\partial y_{i}}\right\| _{p}^{p}\right) ^{1/p}
\end{equation*}%
considered as defined on $\mathcal{B}_{A}^{1,p}$, is a norm on the subspace $%
\mathcal{B}_{A}^{1,p}/\mathbb{R}$ of $\mathcal{B}_{A}^{1,p}$ consisting of
functions $u\in \mathcal{B}_{A}^{1,p}$ with $M(u)=0$. Unfortunately, under
this norm, $\mathcal{B}_{A}^{1,p}/\mathbb{R}$ is a normed vector space which
is in general not complete. We denote by $\mathcal{B}_{\#A}^{1,p}$ the
completion of $\mathcal{B}_{A}^{1,p}/\mathbb{R}$ with respect to that norm,
and by $J_{1}$ the canonical embedding of $\mathcal{B}_{A}^{1,p}/\mathbb{R}$
into $\mathcal{B}_{\#A}^{1,p}$. By the theory of completion of uniform
spaces \cite[Chap. II]{Bourbaki}, the mapping $\overline{\partial }/\partial
y_{i}:\mathcal{B}_{A}^{1,p}/\mathbb{R}\rightarrow \mathcal{B}_{A}^{p}$
extends by continuity to a unique continuous linear mapping still denoted by
$\overline{\partial }/\partial y_{i}:\mathcal{B}_{\#A}^{1,p}\rightarrow
\mathcal{B}_{A}^{p}$ such that
\begin{equation}
\frac{\overline{\partial }}{\partial y_{i}}\circ J_{1}=\frac{\overline{%
\partial }}{\partial y_{i}}\text{\ and }\left\| u\right\| _{\mathcal{B}%
_{\#A}^{1,p}}=\left\| \overline{D}_{y}u\right\| _{p}\;\;(u\in \mathcal{B}%
_{\#A}^{1,p})  \label{2.4'}
\end{equation}%
where $\overline{D}_{y}=(\overline{\partial }/\partial y_{i})_{1\leq i\leq
N} $. Since $\mathcal{G}_{1}$ is an isometric isomorphism of $\mathcal{B}%
_{A}^{1,p}$ onto $W^{1,p}(\Delta (A))$ we have by the definition of $%
\overline{\partial }/\partial y_{i}$ that the restriction of $\mathcal{G}%
_{1} $ to $\mathcal{B}_{A}^{1,p}/\mathbb{R}$ sends isometrically and
isomorphically $\mathcal{B}_{A}^{1,p}/\mathbb{R}$ onto $W^{1,p}(\Delta (A))/%
\mathbb{R}$. So by \cite[Chap. II]{Bourbaki} there exists a unique isometric
isomorphism $\overline{\mathcal{G}}_{1}:\mathcal{B}_{\#A}^{1,p}\rightarrow
W_{\#}^{1,p}(\Delta (A))$ such that
\begin{equation}
\overline{\mathcal{G}}_{1}\circ J_{1}=J\circ \mathcal{G}_{1}\;\;\;\;\;\;\;\;%
\;\;\;\;\;\;\;\;\;\;\;\;\;  \label{2.5'}
\end{equation}%
and
\begin{equation}
\partial _{i}\circ \overline{\mathcal{G}}_{1}=\mathcal{G}_{1}\circ \frac{%
\overline{\partial }}{\partial y_{i}}\;\;(1\leq i\leq N).\;\;\;\;\;\;\;\;\;
\label{2.6'}
\end{equation}%
We recall that in this case (when $A$ is ergodic), $J$ is the canonical
embedding of $W^{1,p}(\Delta (A))/\mathbb{R}$ into its completion $%
W_{\#}^{1,p}(\Delta (A))$ while $J_{1}$ is the canonical embedding of $%
\mathcal{B}_{A}^{1,p}/\mathbb{R}$ into $\mathcal{B}_{\#A}^{1,p}$.
Furthermore, Since $\mathcal{B}_{A}^{1,p}/\mathbb{R}$ is dense in $\mathcal{B%
}_{\#A}^{1,p}$ (in fact by the embedding $J_{1}$, $\mathcal{B}_{A}^{1,p}/%
\mathbb{R}$ is viewed as a subspace of $\mathcal{B}_{\#A}^{1,p}$, and by the
theory of completion, $J_{1}(\mathcal{B}_{A}^{1,p}/\mathbb{R})$ is dense in $%
\mathcal{B}_{\#A}^{1,p}$), it follows that, as $A^{\infty }$ is dense in $A$%
, $\varrho (A^{\infty }/\mathbb{R})$ is dense in $\mathcal{B}_{\#A}^{1,p}$,
where $A^{\infty }/\mathbb{R}=\{u\in A^{\infty }:M(u)=0\}$.

\begin{remark}
\label{r2.1}\emph{For }$u\in B_{A}^{1,p}$\emph{\ (that is the space of }$%
u\in B_{A}^{p}$\emph{\ such that }$D_{y}u\in (B_{A}^{p})^{N}$\emph{) we have
}%
\begin{equation*}
\mathcal{G}_{1}\left( \varrho \left( \frac{\partial u}{\partial y_{i}}%
\right) \right) =\mathcal{G}\left( \frac{\partial u}{\partial y_{i}}\right)
=\partial _{i}\mathcal{G}\left( u\right) =\partial _{i}\mathcal{G}_{1}\left(
\varrho (u)\right) =(\text{\emph{by definition}})\,\mathcal{G}_{1}\left(
\frac{\overline{\partial }}{\partial y_{i}}(\varrho (u))\right) ,
\end{equation*}%
\emph{hence }%
\begin{equation*}
\varrho \left( \frac{\partial u}{\partial y_{i}}\right) =\frac{\overline{%
\partial }}{\partial y_{i}}(\varrho
(u)),\;\;\;\;\;\;\;\;\;\;\;\;\;\;\;\;\;\;\;\;\;\;
\end{equation*}%
\emph{or equivalently, }%
\begin{equation}
\varrho \circ \frac{\partial }{\partial y_{i}}=\frac{\overline{\partial }}{%
\partial y_{i}}\circ \varrho \text{\ \emph{on }}B_{A}^{1,p}.  \label{2.3'}
\end{equation}
\end{remark}

\begin{remark}
\label{r2.2}\emph{The above remark shows that }$\overline{\partial }%
/\partial y_{i}$\emph{, viewed as defined on }$\mathcal{B}_{A}^{p}$\emph{,
is in fact the infinitesimal generator of the group of transformations }$%
T(y) $\emph{\ defined on }$\mathcal{B}_{A}^{p}$\emph{\ by }%
\begin{equation*}
T(y)(u+\mathcal{N})=u(\cdot +y)+\mathcal{N}.
\end{equation*}%
\emph{This shows that all the above results can be obtained through the
theory of strongly continuous groups as shown in \cite{Woukeng} (see also
\cite{WoukengArxiv}).}
\end{remark}

\section{The $\Sigma $-convergence method for stochastic processes}

In this section we define an appropriate notion of the concept of $\Sigma $%
-convergence adapted to our situation. It is to be noted that it is built
according to the original notion introduced by Nguetseng \cite{Hom1}. Here
we adapt it to systems involving random behavior. In all that follows, $Q$
is an open subset of $\mathbb{\mathbb{R}}^{N}$ (integer $N\geq 1$), $T$ is a
positive real number and $Q_{T}=Q\times (0,T)$. Let $(\Omega ,\mathcal{F},%
\mathbb{P})$ be a probability space. The expectation on $(\Omega ,\mathcal{F}%
,\mathbb{P})$ will throughout be denoted by $\mathbb{E}$. Let us first
recall the definition of the Banach space of bounded $\mathcal{F}$%
-measurable functions. Denoting by $F(\Omega )$ the Banach space of all
bounded functions $f:\Omega \rightarrow \mathbb{R}$ (with the sup norm), we
define $B(\Omega )$ as the closure in $F(\Omega )$ of the vector space $%
H(\Omega )$ consisting of all finite linear combinations of the
characteristic functions $1_{X}$ of sets $X\in \mathcal{F}$. Since $\mathcal{%
F}$ is an $\sigma $-algebra, $B(\Omega )$ is the Banach space of all bounded
$\mathcal{F}$-measurable functions. Likewise we define the space $B(\Omega
;Z)$ of all bounded $(\mathcal{F},B_{Z})$-measurable functions $f:\Omega
\rightarrow Z$, where $Z$ is a Banach space endowed with the $\sigma $%
-algebra of Borelians $B_{Z}$. The tensor product $B(\Omega )\otimes Z$ is a
dense subspace of $B(\Omega ;Z)$: this follows from the obvious fact that $%
B(\Omega )$ can be viewed as a space of continuous functions over the
\textit{gamma-compactification} \cite{Zhdanok} of the measurable space $%
(\Omega ,\mathcal{F})$, which is a compact topological space. Next, for $X$
a Banach space, we denote by $L^{p}(\Omega ,\mathcal{F},\mathbb{P};X)$ the
space of $X$-valued random variables $u$ such that $\left\Vert u\right\Vert
_{X}$ is $L^{p}(\Omega ,\mathcal{F},\mathbb{P})$-integrable.

This being so, let $A_{y}$ and $A_{\tau }$ be two algebras wmv on $\mathbb{R}%
_{y}^{N}$ and $\mathbb{R}_{\tau }$ respectively, and let $A=A_{y}\odot
A_{\tau }$ be their product as defined in the preceding section. We know
that $A$ is the closure in $BUC(\mathbb{R}_{y,\tau }^{N+1})$ of the tensor
product $A_{y}\otimes A_{\tau }$. We denote by $\Delta (A_{y})$ (resp. $%
\Delta (A_{\tau })$, $\Delta (A)$) the spectrum of $A_{y}$ (resp. $A_{\tau }$%
, $A$). The same letter $\mathcal{G}$ will denote the Gelfand transformation
on $A_{y}$, $A_{\tau }$ and $A$, as well. Points in $\Delta (A_{y})$ (resp. $%
\Delta (A_{\tau })$) are denoted by $s$ (resp. $s_{0}$). The $M$-measure on
the compact space $\Delta (A_{y})$ (resp. $\Delta (A_{\tau })$) is denoted
by $\beta _{y}$ (resp. $\beta _{\tau }$). We have $\Delta (A)=\Delta
(A_{y})\times \Delta (A_{\tau })$ (Cartesian product) and the $M$-measure on
$\Delta (A)$ is precisely the product measure $\beta =\beta _{y}\otimes
\beta _{\tau }$; the last equality follows in an obvious way by the density
of $A_{y}\otimes A_{\tau }$ in $A$ and by the Fubini's theorem. Points in $%
\Omega $ are as usual denoted by $\omega $.

Unless otherwise stated, random variables will always be considered on the
probability space $(\Omega ,\mathcal{F},\mathbb{P})$. Finally, the letter $E$
will throughout denote exclusively an ordinary sequence $(\varepsilon
_{n})_{n\in \mathbb{N}}$ with $0<\varepsilon _{n}\leq 1$ and $\varepsilon
_{n}\rightarrow 0$ as $n\rightarrow \infty $. In what follows, the notations
are those of the preceding section.

\begin{definition}
\label{d3.1}\emph{A sequence of random variables }$(u_{\varepsilon
})_{\varepsilon >0}\subset L^{p}(\Omega ,\mathcal{F},\mathbb{P}%
;L^{p}(Q_{T})) $\emph{\ (}$1\leq p<\infty $\emph{) is said to }weakly $%
\Sigma $-converge\emph{\ in }$L^{p}(Q_{T}\times \Omega )$\emph{\ to some
random variable }$u_{0}\in L^{p}(\Omega ,\mathcal{F},\mathbb{P};L^{p}(Q_{T};%
\mathcal{B}_{A}^{p}))$\emph{\ if as }$\varepsilon \rightarrow 0$\emph{, we
have }
\begin{equation}
\begin{array}{l}
\int_{Q_{T}\times \Omega }u_{\varepsilon }(x,t,\omega )f\left( x,t,\frac{x}{%
\varepsilon },\frac{t}{\varepsilon },\omega \right) dxdtd\mathbb{P} \\
\ \ \ \ \ \ \ \ \ \ \ \ \ \ \rightarrow \iint_{Q_{T}\times \Omega \times
\Delta (A)}\widehat{u}_{0}(x,t,s,s_{0},\omega )\widehat{f}%
(x,t,s,s_{0},\omega )dxdtd\mathbb{P}d\beta%
\end{array}
\label{3.1}
\end{equation}%
\emph{for every }$f\in L^{p^{\prime }}(\Omega ,\mathcal{F},\mathbb{P}%
;L^{p^{\prime }}(Q_{T};A))$\emph{\ (}$1/p^{\prime }=1-1/p$\emph{), where }$%
\widehat{u}_{0}=\mathcal{G}_{1}\circ u_{0}$\emph{\ and }$\widehat{f}=%
\mathcal{G}_{1}\circ (\varrho \circ f)=\mathcal{G}\circ f$\emph{. We express
this by writing} $u_{\varepsilon }\rightarrow u_{0}$ in $L^{p}(Q_{T}\times
\Omega )$-weak $\Sigma $.
\end{definition}

\begin{remark}
\label{r3.0}\emph{The above weak }$\Sigma $\emph{-convergence in }$%
L^{p}(Q_{T}\times \Omega )$\emph{\ implies the weak convergence in }$%
L^{p}(Q_{T}\times \Omega )$\emph{. One can show as in the usual setting of }$%
\Sigma $\emph{-convergence method \cite{Hom1} that each }$f\in L^{p}(\Omega ,%
\mathcal{F},\mathbb{P};L^{p}(Q_{T};A))$\emph{\ weakly }$\Sigma $\emph{%
-converges to }$\varrho \circ f$\emph{.}
\end{remark}

In order to simplify the notation, we will henceforth denote $L^{p}(\Omega ,%
\mathcal{F},\mathbb{P};X)$ merely by $L^{p}(\Omega ;X)$ if it is understood
from the context and there is no danger of confusion. Definition \ref{d3.1}
can be formally motivated by the following fact. Assume $p=2$; then using
the chaos decomposition (see \cite{Cameron, Wiener})\ of $u_{\varepsilon }$
and $f$ we get $u_{\varepsilon }(x,t,\omega )=\sum_{j=1}^{\infty
}u_{\varepsilon ,j}(x,t)\Phi _{j}(\omega )$ and $f(x,t,y,\tau ,\omega
)=\sum_{k=1}^{\infty }f_{k}(x,t,y,\tau )\Phi _{k}(\omega )$ where $%
u_{\varepsilon ,j}\in L^{2}(Q_{T})$ and $f_{k}\in L^{2}(Q_{T};A)$, so that
\begin{equation*}
\int_{Q_{T}\times \Omega }u_{\varepsilon }(x,t,\omega )f\left( x,t,\frac{x}{%
\varepsilon },\frac{t}{\varepsilon },\omega \right) dxdtd\mathbb{P}
\end{equation*}%
can be formally written as
\begin{equation*}
\sum_{j,k}\int_{\Omega }\Phi _{j}(\omega )\Phi _{k}(\omega )d\mathbb{P}%
\int_{Q_{T}}u_{\varepsilon ,j}(x,t)f_{k}\left( x,t,\frac{x}{\varepsilon },%
\frac{t}{\varepsilon }\right) dxdt,
\end{equation*}%
and by the usual $\Sigma $-convergence method (see \cite{NA, Hom1}), as $%
\varepsilon \rightarrow 0$,
\begin{equation*}
\int_{Q_{T}}u_{\varepsilon ,j}(x,t)f_{k}\left( x,t,\frac{x}{\varepsilon },%
\frac{t}{\varepsilon }\right) dxdt\rightarrow \iint_{Q_{T}\times \Delta (A)}%
\widehat{u}_{0,j}(x,t,s,s_{0})\widehat{f}_{k}\left( x,t,s,s_{0}\right)
dxdtd\beta .
\end{equation*}%
Hence, by setting
\begin{equation*}
\widehat{u}_{0}(x,t,s,s_{0},\omega )=\sum_{j=1}^{\infty }\widehat{u}%
_{0,j}(x,t,s,s_{0})\Phi _{j}(\omega );\;\widehat{f}\left( x,t,s,s_{0},\omega
\right) =\sum_{k=1}^{\infty }\widehat{f}_{k}\left( x,t,s,s_{0}\right) \Phi
_{k}(\omega )
\end{equation*}%
we get (\ref{3.1}). We can also see that (\ref{3.1}) is a straight
generalization of the usual concept of $\Sigma $-convergence.

The following result holds.

\begin{theorem}
\label{t3.1}Let $1<p<\infty $. Let $(u_{\varepsilon })_{\varepsilon \in
E}\subset L^{p}(\Omega ;L^{p}(Q_{T}))$ be a sequence of random variables
verifying the following boundedness condition:
\begin{equation*}
\sup_{\varepsilon \in E}\mathbb{E}\left\Vert u_{\varepsilon }\right\Vert
_{L^{p}(Q_{T})}^{p}<\infty .
\end{equation*}%
Then there exists a subsequence $E^{\prime }$ from $E$ such that the
sequence $(u_{\varepsilon })_{\varepsilon \in E^{\prime }}$ is weakly $%
\Sigma $-convergent in $L^{p}(Q_{T}\times \Omega )$.
\end{theorem}

\begin{proof}
Let us set $Y=L^{p^{\prime }}(Q_{T}\times \Omega \times \Delta (A))$ and $%
X=L^{p^{\prime }}(\Omega ;L^{p^{\prime }}(Q_{T};\mathcal{C}(\Delta (A))))=%
\mathcal{G}(L^{p^{\prime }}(\Omega ;L^{p^{\prime }}(Q_{T};A))).$ Applying
\cite[Theorem 3.1]{CMP} with $Y$ and $X$ we are led at once to the result.
\end{proof}

The following result will be very useful in the homogenization process.

\begin{theorem}
\label{t3.2}Let $1<p<\infty $. Let $A=A_{y}\odot A_{\tau }$ be an algebra
wmv on $\mathbb{R}_{y}^{N}\times \mathbb{R}_{\tau }$ with the further
property that $A_{y}$ is ergodic. Finally let $(u_{\varepsilon
})_{\varepsilon \in E}\subset L^{p}(\Omega ;L^{p}(0,T;W_{0}^{1,p}(Q)))$ be a
sequence of random variables which satisfies the following estimate:
\begin{equation*}
\sup_{\varepsilon \in E}\mathbb{E}\left\| u_{\varepsilon }\right\|
_{L^{p}(0,T;W_{0}^{1,p}(Q))}^{p}<\infty .
\end{equation*}%
Then there exist a subsequence $E^{\prime }$ of $E$ and a couple of random
variables $(u_{0},u_{1})$ with $u_{0}\in L^{p}(\Omega
;L^{p}(0,T;W_{0}^{1,p}(Q)))$ and $u_{1}\in L^{p}(\Omega ;L^{p}(Q_{T};%
\mathcal{B}_{A_{\tau }}^{p}(\mathbb{R}_{\tau };\mathcal{B}%
_{\#A_{y}}^{1,p}))) $ such that, as $E^{\prime }\ni \varepsilon \rightarrow
0 $,
\begin{equation}
u_{\varepsilon }\rightarrow u_{0}\ \text{in }L^{p}(Q_{T}\times \Omega )\text{%
-weak;}  \label{3.0}
\end{equation}%
\begin{equation}
\frac{\partial u_{\varepsilon }}{\partial x_{i}}\rightarrow \frac{\partial
u_{0}}{\partial x_{i}}+\frac{\overline{\partial }u_{1}}{\partial y_{i}}\text{%
\ in }L^{p}(Q_{T}\times \Omega )\text{-weak }\Sigma \text{, }1\leq i\leq N.
\label{3.2}
\end{equation}
\end{theorem}

\begin{proof}
The proof of the above theorem follows exactly the same lines of reasoning
as the one of \cite[Theorem 3.6]{NA}.
\end{proof}

In practice, we will mostly deal with the following modified version of the
above theorem.

\begin{theorem}
\label{t3.3}Assume that the hypotheses of Theorem \emph{\ref{t3.2}} are
satisfied. Assume further that $p\geq 2$ and that there exist a subsequence $%
E^{\prime }$ from $E$ and a random variable $u_{0}\in L^{p}(\Omega
;L^{p}(0,T;W_{0}^{1,p}(Q)))$ such that, as $E^{\prime }\ni \varepsilon
\rightarrow 0$,
\begin{equation}
u_{\varepsilon }\rightarrow u_{0}\text{\ in }L^{2}(Q_{T}\times \Omega )\text{%
.}  \label{3.3}
\end{equation}%
Then there exist a subsequence of $E^{\prime }$ (not relabeled) and a $%
\mathcal{B}_{A_{\tau }}^{p}(\mathbb{R}_{\tau };\mathcal{B}_{\#A_{y}}^{1,p})$%
-valued stochastic process $u_{1}\in L^{p}(\Omega ;L^{p}(Q_{T};\mathcal{B}%
_{A_{\tau }}^{p}(\mathbb{R}_{\tau };\mathcal{B}_{\#A_{y}}^{1,p})))$ such
that \emph{(\ref{3.2})} holds when $E^{\prime }\ni \varepsilon \rightarrow 0$%
.
\end{theorem}

\begin{proof}
Since $(u_{\varepsilon })_{\varepsilon \in E^{\prime }}$ is bounded in $%
L^{p}(\Omega ;L^{p}(0,T;W_{0}^{1,p}(Q)))$, there exist a subsequence of $%
E^{\prime }$ not relabeled and $v_{0}\in L^{p}(\Omega
;L^{p}(0,T;W_{0}^{1,p}(Q)))$ such that $u_{\varepsilon }\rightarrow v_{0}$
in $L^{p}(Q_{T}\times \Omega )$-weak (and hence in $L^{2}(Q_{T}\times \Omega
)$-weak since $p\geq 2$) as $E^{\prime }\ni \varepsilon \rightarrow 0$. From
(\ref{3.3}) and owing to the uniqueness of the weak-limit, we infer that $%
u_{0}=v_{0}$, so that (\ref{3.0}) holds true with $u_{0}$ as in (\ref{3.3}).
The remainder of the proof follows exactly the same lines of reasoning as in
the proof of \cite[Theorem 3.6]{NA}.
\end{proof}

\section{Statement of the problem: a priori estimates and tightness property}

\subsection{Problem setting}

Let $(\Omega ,\mathcal{F},\mathbb{P})$ be a probability space on which is
defined an infinite sequence of independent standard 1-d Brownian motion $%
(W_{k})_{k\geq 1}$. We equip the probability space by the natural
filtration, denoted by $\mathcal{F}^{t}$, of $W_{k}$. Now let $\mathcal{U}$
be a fixed Hilbert space with orthonormal basis $\left\{ e_{k}:k\geq
1\right\} $. We may define a cylindrical Wiener process $W$ by setting $%
W=\sum_{k=1}^{\infty }W_{k}e_{k}$ (see \cite{DaPrato}). By $L_{2}(\mathcal{U}%
,X)$ we denote the space of Hilbert-Schmidt operators from $\mathcal{U}$ to
the Hilbert space $X$:
\begin{equation*}
L_{2}(\mathcal{U},X)=\left\{ R\in L(\mathcal{U},X):\sum_{k=1}^{\infty
}|Re_{k}|_{X}^{2}<\infty \right\} .
\end{equation*}%
We can define another Hilbert space $\mathcal{U}_{0}\subset \mathcal{U}$ by
setting
\begin{equation*}
\mathcal{U}_{0}=\left\{ v=\sum_{k=1}^{\infty }\alpha
_{k}e_{k}:\sum_{k=1}^{\infty }\alpha _{k}^{2}k^{-2}<\infty \right\} .
\end{equation*}%
Note that the embedding $\mathcal{U}_{0}\subset \mathcal{U}$ is
Hilbert-Schmidt. We endow $\mathcal{U}_{0}$ with the norm $\left\vert
v\right\vert _{\mathcal{U}_{0}}^{2}=\sum_{k=1}^{\infty }\alpha
_{k}^{2}k^{-2} $. It is a well known fact that there exists $\Omega ^{\prime
}\in \mathcal{F}$ with $\mathbb{P}(\Omega ^{\prime })=1$ such that $W(\omega
)\in \mathcal{C}(0,T;\mathcal{U}_{0})$ for any $\omega \in \Omega ^{\prime }$
(see, for example, \cite{DaPrato}).

For any given $G\in L^{2}(\Omega ;L^{2}(0,T;L_{2}(\mathcal{U},X))$ such that
$G(t)$ is $\mathcal{F}^{t}$-adapted we may define the stochastic integral
\begin{equation*}
\int_{0}^{t}GdW=\sum_{k=1}^{\infty }\int_{0}^{t}Ge_{k}dW_{k},
\end{equation*}%
as an element of the space of $X$-valued square integrable martingale.
Moreover we have
\begin{equation*}
\mathbb{E}\sup_{t\in \lbrack 0,T]}\left| \int_{0}^{t}GdW\right| ^{r}\leq C%
\mathbb{E}\left( \int_{0}^{T}\left| G\right| _{L_{2}(\mathcal{U}%
,X)}^{2}\right) ^{\frac{r}{2}},
\end{equation*}%
for any $r\geq 1$. For the two results and more details on stochastic
calculus in infinite dimension we refer to \cite{DaPrato}. From now we will
set $\left| G\right| _{L_{2}}=\left| G\right| _{L_{2}(\mathcal{U},X)}$ for
any Hilbert space $X$ and for any $G\in L_{2}(\mathcal{U},X)$.

Let $Q\subset \mathbb{R}^{N}$ be an open and bounded domain with smooth
boundary. Throughout we will set $H=L^{2}(Q)$, $V=W_{0}^{1,p}(Q)$ and denote
by $\left\vert u\right\vert ,u\in H$, $\left\Vert v\right\Vert ,v\in V$
their respective norms. We will also denote by $\left\vert \nu \right\vert $%
, $\nu \in \mathbb{R}^{N}$ the Euclidian norm on $\mathbb{R}^{N}$. The
symbol $V^{\prime }$ will denote the dual of $V$ and $\left\langle
u,v\right\rangle $ denotes the duality pairing between $u\in V^{\prime }$
and $v\in V$. The inner product in $H$ is denoted by $(u,v)$ for any $u,v\in
H$. In this work we are interested in the asymptotic behaviour as $%
\varepsilon \rightarrow 0$ of the solution of (\ref{1.1}) which is defined
on the stochastic system $(\Omega ,\mathcal{F},\mathbb{P}),\mathcal{F}^{t},W$%
. We assume that all the coefficients in (\ref{1.1}) are measurable with
respect to each of their arguments. Furthermore, for a.e $(x,t)\in Q\times
(0,T)$, $(y,\tau )\in \mathbb{R}^{N}\times \mathbb{R}$ and for all $\mu \in
\mathbb{R}$ and $\lambda \in \mathbb{R}^{N}$, we assume that

\begin{itemize}
\item[\textbf{A1.}] $a(x,t,y,\tau,\mu,0)=0$,

\item[\textbf{A2.}] $(a(x,t,y,\tau ,\mu ,\lambda )-a(x,t,y,\tau ,\mu
,\lambda ^{\prime })\cdot (\lambda -\lambda ^{\prime }))\geq c_1 \left|
\lambda -\lambda ^{\prime }\right| ^{p}$,

\item[\textbf{A3.}] $\left| a(x,t,y,\tau ,\mu ,\lambda )\right| \leq
c_{2}(1+\left| \mu \right| ^{p-1}+\left| \lambda \right| ^{p-1})$,

\item[\textbf{A4.}] $\left| a_{0}(x,t,y,\tau ,\mu )\right| \leq
c_{3}(1+\left| \mu \right| )$,

\item[\textbf{A5.}] $\left| a_{0}(x,t,y,\tau ,\mu )-a_{0}(x,t,y,\tau ,\mu
^{\prime })\right| \leq c_{4}\left| \mu -\mu ^{\prime }\right| ,$

\begin{itemize}
\item[\textbf{A6}. \textbf{(a)}] $\left\vert a_{0}(x,t,y,\tau
,u)-a_{0}(x^{\prime },t^{\prime },y,\tau ,u^{\prime })\right\vert \leq
m(\left\vert x-x^{\prime }\right\vert +\left\vert t-t^{\prime }\right\vert
+\left\vert u-u^{\prime }\right\vert )(1+\left\vert u\right\vert +\left\vert
u^{\prime }\right\vert )$,

\item[\textbf{A6}. \textbf{(b)}] $\left\vert a(x,t,y,\tau ,u,\mathbf{v}%
)-a(x^{\prime },t^{\prime },y,\tau ,u^{\prime },\mathbf{v}^{\prime
})\right\vert \leq m(\left\vert x-x^{\prime }\right\vert +\left\vert
t-t^{\prime }\right\vert +\left\vert u-u^{\prime }\right\vert
^{p-1}+\left\vert u^{\prime }\right\vert ^{p-1}+\left\vert \mathbf{v}%
\right\vert ^{p-1}+\left\vert \mathbf{v}^{\prime }\right\vert
^{p-1})+C(1+\left\vert u\right\vert +\left\vert u^{\prime }\right\vert
+\left\vert \mathbf{v}\right\vert +\left\vert \mathbf{v}^{\prime
}\right\vert )^{p-2}\left\vert \mathbf{v}-\mathbf{v}^{\prime }\right\vert ,$
\end{itemize}

where $m$ is a continuity modulus (i.e., a nondecreasing continuous function
on $[0,+\infty)$ such that $m(0) = 0, m(r) > 0$ if $r > 0$, and $m(r) = 1$
if $r > 1$).
\end{itemize}

As far as the operator $M$ is concerned, we will suppose that

\begin{itemize}
\item[\textbf{A7.}] it is a measurable mapping from $\mathbb{R}^{N}\times
\mathbb{R}\times H$ into $L_{2}(\mathcal{U},H)$ such that

\begin{itemize}
\item[(a)] $\left\vert M(y,\tau ,u)-M(y,\tau ,v)\right\vert _{L_{2}}\leq
c_{6}\left\vert u-v\right\vert ,$

\item[(b)] $\left\vert M(y,\tau ,u)\right\vert _{L_{2}}\leq
c_{7}(1+\left\vert u\right\vert ).$
\end{itemize}
\end{itemize}

We note that an example of nontrivial functions $a$, $a_{0}$ and $M$
satisfying \textbf{A1}.-\textbf{A7}. are the functions $a(x,t,y,\tau ,\mu
,\lambda )=g(x,t,y,\tau )\left\vert \lambda \right\vert ^{p-2}\lambda $, $%
a_{0}(x,t,y,\tau ,\mu )=g_{0}(x,t,y,\tau )h(\mu )$, $M(y,\tau
,u)=(M_{k}(y,\tau ,u))_{k\geq 1}$ with $M_{k}(y,\tau ,u)=g_{1}(y,\tau
)\lambda _{k}u\geq 0$, where $\sum_{k=1}^{\infty }\left\vert \lambda
_{k}\right\vert ^{2}<\infty $, $g$, $g_{0}\in \mathcal{C}(\overline{Q}_{T};%
\mathcal{B}(\mathbb{R}_{y,\tau }^{N+1}))$, $g_{1}\in \mathcal{B}(\mathbb{R}%
_{y,\tau }^{N+1})$ and $h$ is a continuous Lipschitz function on $\mathbb{R}$%
. We recall that $\mathcal{B}(\mathbb{R}_{y,\tau }^{N+1})$ is the space of
bounded continuous real-valued functions defined on $\mathbb{R}_{y,\tau
}^{N+1}$.

Note that \textbf{A1}.-\textbf{A3}. imply that $A(x,t,y,\tau ,u,Du)\equiv -%
\Div a(x,t,y,\tau ,u,Du)$ satisfies

\begin{enumerate}
\item[\textbf{C1.}]
\begin{eqnarray*}
&&\left\langle A(x,t,y,\tau ,u,Du)-A(x,t,y,\tau ,v,Dv),u-v\right\rangle \\
&\geq &\int_{Q}(a(x,t,y,\tau ,u,Du)-a(x,t,y,\tau ,v,Dv))\cdot (u-v)dx,
\end{eqnarray*}

\item[\textbf{C2.}] $\left\langle A(x,t,y,\tau ,u,Du),u\right\rangle \geq
c_{1}\left\vert Du\right\vert ^{p}$,

\item[\textbf{C3.}] $\left\| A(x,t,y,\tau ,u,Du)\right\| _{W^{-1,p^{\prime
}}(Q)}^{p^{\prime }}\leq c_{2}^{\prime }(1+\left| Du\right| ^{p})$ for some
positive constant $c_{2}^{\prime }$ depending only on $Q_{T}$ and on $c_{2}$,

\item[\textbf{C4.}] the mapping $\theta \rightarrow \left\langle
A(x,t,y,\tau ,u+\theta v,D(u+\theta v)),w\right\rangle :\mathbb{R}%
\rightarrow \mathbb{R}$ is a continuous function for any $u,v,w\in V$.
\end{enumerate}

To simplify the notations we will set throughout
\begin{equation*}
a^{\varepsilon }(\cdot ,u_{\varepsilon },Du_{\varepsilon })(x,t)=a\left( x,t,%
\frac{x}{\varepsilon },\frac{t}{\varepsilon },u_{\varepsilon
},Du_{\varepsilon }\right) ,
\end{equation*}%
\begin{equation*}
a_{0}^{\varepsilon }(\cdot ,u_{\varepsilon })(x,t)=a_{0}\left( x,t,\frac{x}{%
\varepsilon },\frac{t}{\varepsilon },u_{\varepsilon }\right) ,
\end{equation*}%
\begin{equation*}
M^{\varepsilon }(\cdot ,u_{\varepsilon })(x,t)=M\left( \frac{x}{\varepsilon }%
,\frac{t}{\varepsilon },u_{\varepsilon }\right) ,
\end{equation*}%
and
\begin{equation*}
A^{\varepsilon }(\cdot ,u_{\varepsilon },Du_{\varepsilon })(x,t)=A\left( x,t,%
\frac{x}{\varepsilon },\frac{t}{\varepsilon },u_{\varepsilon
},Du_{\varepsilon }\right) .
\end{equation*}%
It is to be noted that the just defined functions make sense as trace
functions; see e.g. \cite{NA, AMPA} for the justification. By a strong
probabilistic solution of (\ref{1.1}) we mean an $\mathcal{F}^{t}$-adapted
stochastic process $u_{\varepsilon }$ such that:
\begin{equation*}
u_{\varepsilon }\in L^{p}(\Omega ,\mathcal{F},\mathbb{P};L^{p}(0,T;V))\cap
L^{2}(\Omega ,\mathcal{F},\mathbb{P};\mathcal{C}(0,T;H)),
\end{equation*}%
and for all $\phi \in V$ and for almost every $(\omega ,t)\in \Omega \times
\lbrack 0,T]$ the following holds true
\begin{equation*}
\begin{split}
(u_{\varepsilon }(t),\phi )+\int_{0}^{t}(a^{\varepsilon }(\cdot
,u_{\varepsilon }(s),Du_{\varepsilon }(s)),D\phi )ds& =(u^{0},\phi
)-\int_{0}^{t}(a_{0}^{\varepsilon }(\cdot ,u_{\varepsilon }(s)),\phi )ds \\
& +\sum_{k=1}^{\infty }\int_{0}^{t}(M_{k}^{\varepsilon }(\cdot
,u_{\varepsilon }(s)),\phi )dW_{k},
\end{split}%
\end{equation*}%
where $M_{k}^{\varepsilon }(\cdot ,u_{\varepsilon }(s))=M^{\varepsilon
}(\cdot ,u_{\varepsilon }(s))e_{k}$. Under the above conditions, it is
easily seen that if $u_{\varepsilon }$ and $v_{\varepsilon }$ are two
solutions to (\ref{1.1}) on the same stochastic system $(\Omega ,\mathcal{F},%
\mathbb{P}),\mathcal{F}^{t},W$ with the same initial condition $u^{0}$, then
$u_{\varepsilon }(t)=v_{\varepsilon }(t)$ in $H$ almost surely for any $t$.
Thanks to this fact together with the Yamada-Watanabe's Theorem (see \cite%
{revuz}) and the existence result of martingale solutions in \cite%
{bensoussan2}, we see that (\ref{1.1}) has a unique strong probabilistic
solution.

\subsection{The \textit{a} priori estimates}

Throughout $C$ will denote a generic constant independent of $\varepsilon $.
We have the following result whose proof can be obtained in a standard way.

\begin{lemma}
\label{l0:1}The solution $u_{\varepsilon }$ of \emph{(\ref{1.1})} satisfies
the following inequalities
\begin{align}
\mathbb{E}\sup_{0\leq t\leq T}\left\vert u_{\varepsilon }(t)\right\vert
^{4}& \leq C,  \label{0:3} \\
\mathbb{E}\int_{0}^{T}\left\vert Du_{\varepsilon }(t)\right\vert ^{p}dt&
\leq C.  \label{0:4}
\end{align}
\end{lemma}

The following result is very crucial for the proof of the tightness property
of $u_{\varepsilon }$.

\begin{lemma}
\label{l0.2}There exists a constant $C>0$ such that
\begin{equation*}
\mathbb{E}\sup_{\left\vert \theta \right\vert \leq \delta
}\int_{0}^{T}\left\vert u_{\varepsilon }(t+\theta )-u_{\varepsilon
}(t)\right\vert _{V^{\prime }}^{p^{\prime }}dt\leq C\delta ^{\frac{1}{p-1}},
\end{equation*}%
for any $\varepsilon $, and $\delta \in (0,1)$. Here we assume that $%
u_{\varepsilon }(t)$ has zero extension outside the interval $[0,T]$.
\end{lemma}

\begin{proof}
Let us assume $\theta \geq 0$ (as we will see in what follows the same
argument will apply for $\theta <0$). We will denote by $p^{\prime }$ the H%
\"{o}lder conjugate of $p$ (i.e., $\frac{1}{p}+\frac{1}{p^{\prime }}=1$). It
is clear that
\begin{equation}
\begin{split}
\left\vert u_{\varepsilon }(t+\theta )-u_{\varepsilon }(t)\right\vert
_{V^{\prime }}^{p^{\prime }}& \leq C\left\vert \int_{t}^{t+\theta
}A^{\varepsilon }(\cdot ,u_{\varepsilon }(s),Du_{\varepsilon
}(s))ds\right\vert _{V^{\prime }}^{p^{\prime }}+C\left\vert
\int_{t}^{t+\theta }a_{0}^{\varepsilon }(\cdot ,u_{\varepsilon
}(s))ds\right\vert _{V^{\prime }}^{p^{\prime }} \\
& +C\left\vert \int_{0}^{t}M^{\varepsilon }(\cdot ,u_{\varepsilon
}(s))dW\right\vert _{V^{\prime }}^{p^{\prime }}. \\
& \leq I_{1}(t,\theta )+I_{2}(t,\theta )+I_{3}(t,\theta ).
\end{split}
\label{0:13}
\end{equation}%
It is not difficult to show that
\begin{equation*}
I_{1}(t,\theta )\leq \theta ^{\frac{p^{\prime }}{p}}\int_{t}^{t+\theta
}\left\vert A^{\varepsilon }(\cdot ,u_{\varepsilon }(s),Du_{\varepsilon
}(s))\right\vert ^{p^{\prime }}ds.
\end{equation*}%
Therefore
\begin{equation*}
\mathbb{E}\sup_{\theta \leq \delta }\int_{0}^{T}I_{1}(t,\theta )dt\leq
C\delta ^{\frac{p^{\prime }}{p}}\mathbb{E}\int_{0}^{T}\int_{t}^{t+\delta
}\left\vert Du_{\varepsilon }\right\vert ^{p}ds.
\end{equation*}%
Thanks to \eqref{0:4} we have that
\begin{equation}
\mathbb{E}\sup_{\theta \leq \delta }\int_{0}^{T}I_{1}(t,\theta )dt\leq
C\delta ^{\frac{p^{\prime }}{p}}.  \label{0:14}
\end{equation}%
Thanks to \textbf{A4.}, we have that
\begin{equation*}
I_{2}(t,\theta )\leq C\left( \int_{t}^{t+\theta }(1+\left\vert
u_{\varepsilon }(s)\right\vert )ds\right) ^{p^{\prime }},
\end{equation*}%
which implies that
\begin{equation*}
\mathbb{E}\sup_{\theta \leq \delta }\int_{0}^{T}I_{2}(t,\theta )dt\leq
\int_{0}^{T}\mathbb{E}\left( \delta +\int_{t}^{t+\theta }\left\vert
u_{\varepsilon }(s)\right\vert ds\right) ^{p^{\prime }}dt.
\end{equation*}%
We invoke from this that
\begin{equation*}
\mathbb{E}\sup_{\theta \leq \delta }\int_{0}^{T}I_{2}(t,\theta )dt\leq
C\left( \int_{0}^{T}\mathbb{E}\left( \delta +\int_{t}^{t+\delta }\left\vert
u_{\varepsilon }(s)\right\vert ds\right) ^{2}dt\right) ^{\frac{p^{\prime }}{2%
}}.
\end{equation*}%
Thanks to \eqref{0:3} we deduce from this that
\begin{equation}
\mathbb{E}\sup_{\theta \leq \delta }\int_{0}^{T}I_{2}(t,\theta )dt\leq
C\delta ^{p^{\prime }}.  \label{0:15}
\end{equation}%
Next, by using Burkh\"{o}lder-Davis-Gundy's inequality we see that
\begin{equation*}
\begin{split}
\mathbb{E}\sup_{0\leq \theta \leq \delta }\int_{0}^{T}I_{3}(t,\theta )dt&
\leq C\int_{0}^{T}\mathbb{E}\left( \int_{t}^{t+\delta }\left\vert
M^{\varepsilon }(\cdot ,u_{\varepsilon }(s))\right\vert ^{2}ds\right) ^{%
\frac{p^{\prime }}{2}}dt \\
& \leq C\left[ \int_{0}^{T}\mathbb{E}\left( \int_{t}^{t+\delta }\left\vert
M^{\varepsilon }(\cdot ,u_{\varepsilon }(s))\right\vert ^{2}ds\right) ^{2}dt%
\right] ^{\frac{p^{\prime }}{4}}
\end{split}%
\end{equation*}%
By condition \textbf{A7.},
\begin{equation*}
\mathbb{E}\sup_{0\leq \theta \leq \delta }\int_{0}^{T}I_{3}(t,\theta )dt\leq
C\left[ \int_{0}^{T}\delta ^{2}+\mathbb{E}\sup_{s\in \lbrack 0,T]}\left\vert
u_{\varepsilon }(s)\right\vert ^{4}dt\right] ^{\frac{p^{\prime }}{4}},
\end{equation*}%
which clearly implies that
\begin{equation}
\mathbb{E}\sup_{0\leq \theta \leq \delta }\int_{0}^{T}I_{3}(t,\theta )dt\leq
C\delta ^{\frac{p^{\prime }}{2}}.  \label{0:16}
\end{equation}%
Combining \eqref{0:14}, \eqref{0:15} and \eqref{0:16} we infer from %
\eqref{0:13} that
\begin{equation*}
\mathbb{E}\sup_{0\leq \theta \leq \delta }\int_{0}^{T}\left\vert
u_{\varepsilon }(t+\theta )-u_{\varepsilon }(t)\right\vert ^{p^{\prime
}}dt\leq C\delta ^{\frac{p^{\prime }}{p}},
\end{equation*}%
since $\frac{p^{\prime }}{p}\leq 1$. A same inequality holds for $\theta <0$%
. This ends the proof of the lemma.
\end{proof}

\subsection{Tightness property}

To prove the tightness of the law of $(u_\varepsilon, W)$ we will mainly
follow the idea in \cite{bensoussan2} and in \cite{Glatt}. Let us consider
the mappings:
\begin{equation*}
\begin{split}
\psi _{1}^{\varepsilon }& :\omega \in \Omega \mapsto u_{\varepsilon
}(\omega)\in L^{p}(0,T;H) \\
\psi _{2}^{\varepsilon }& :\omega \in \Omega \mapsto W(\omega)\in \mathcal{C}%
(0,T;\mathcal{U}_{0}).
\end{split}%
\end{equation*}%
We denote by $\mathfrak{S}_{1}=L^{p}(0,T,H)$ ( resp., $\mathfrak{S}_{2}=%
\mathcal{C}(0,T;\mathcal{U}_{0})$) and $\mathcal{B}(\mathfrak{S}_{1})$
(resp., $\mathcal{B}(\mathfrak{S}_{2})$) its Borel $\sigma $-algebra. The
mappings
\begin{equation*}
\Pi _{1}^{\varepsilon }(A)=\mathbb{P}\circ \psi _{1}^{\varepsilon }(A)\equiv
\mathbb{P}((\psi _{1}^{\varepsilon })^{-1}(A)),\ A\in \mathcal{B}(\mathfrak{S%
}_{1}),
\end{equation*}%
\begin{equation*}
\Pi _{2}^{\varepsilon }(A)=\mathbb{P}\circ \psi _{2}^{\varepsilon }(A)\equiv
\mathbb{P}((\psi _{2}^{\varepsilon })^{-1}(A)),\ A\in \mathcal{B}(\mathfrak{S%
}_{2}),
\end{equation*}%
and
\begin{equation*}
\Pi ^{\varepsilon }=\Pi _{1}^{\varepsilon }\times \Pi _{2}^{\varepsilon }
\end{equation*}%
define families of probability measures on $(\mathfrak{S}_{1},\mathcal{B}(%
\mathfrak{S}_{1}))$, $(\mathfrak{S}_{2},\mathcal{B}(\mathfrak{S}_{2}))$ and $%
(\mathfrak{S}=\mathfrak{S}_{1}\times \mathfrak{S}_{2},\mathcal{B}(\mathfrak{S%
}_{1}\times \mathfrak{S}_{2}))$, respectively.

\begin{lemma}
\label{l0.3}Let $\mu _{n},\nu _{n}$ be sequences of positive numbers such
that $\mu _{n},\nu _{n}\rightarrow 0$ as $n\rightarrow \infty $. The set
\begin{equation*}
Z=\left\{ z:\int_{0}^{T}\left\Vert Dz\right\Vert ^{p}dt\leq L\text{,}\
\left\vert z(t)\right\vert ^{2}\leq K\text{ a.e. }t\text{,}\
\sup_{\left\vert \theta \right\vert \leq \mu _{n}}\int_{0}^{T}\left\vert
z(t+\theta )-z(t)\right\vert _{V^{\prime }}^{p^{\prime }}\leq \nu
_{n}M\right\}
\end{equation*}%
is a compact subset of $L^{p}(0,T;H)$.
\end{lemma}

\begin{proof}
The proof is the same as in \cite[Proposition 3.1]{bensoussan2}.
\end{proof}

The following result is of great importance for the rest of the work.

\begin{lemma}
\label{l0.4}The family $\Pi ^{\varepsilon }$ is tight on $\mathfrak{S}$.
\end{lemma}

\begin{proof}
Let $\delta >0$ and let $L_{\delta },K_{\delta },M_{\delta }$ positive
constants depending only on $\delta $ to be fixed later. It follows from
Lemma \ref{l0.3} that
\begin{equation*}
Z_{\delta }=\left\{ z:\int_{0}^{T}\left\vert Dz\right\vert ^{p}dt\leq
L_{\delta },\ \left\vert z(t)\right\vert ^{2}\leq K_{\delta }\text{ a.e. }%
t,\ \sup_{\left\vert \theta \right\vert \leq \mu _{n}}\int_{0}^{T}\left\vert
z(t+\theta )-z(t)\right\vert _{V^{\prime }}^{p^{\prime }}\leq \nu
_{n}M_{\delta }\right\}
\end{equation*}%
is a compact subset of $L^{p}(0,T;H)$ for any $\delta >0$. Here we choose
the sequence $\mu _{n},\nu _{n}$ so that $\sum \frac{1}{\nu _{n}}(\mu _{n})^{%
\frac{p^{\prime }}{p}}<\infty $. We have that
\begin{equation*}
\begin{split}
\mathbb{P}\left( u_{\varepsilon }\notin Z_{\delta }\right) & \leq \mathbb{P}%
\left( \int_{0}^{T}\left\vert Du_{\varepsilon }(s)\right\vert ^{p}ds\geq
L_{\delta }\right) +\mathbb{P}\left( \sup_{s\in \lbrack 0,T]}\left\vert
u_{\varepsilon }(s)\right\vert ^{2}\geq K_{\delta }\right) \\
& +\mathbb{P}\left( \sup_{\left\vert \theta \right\vert \leq \mu
_{n}}\int_{0}^{T}\left\vert u_{\varepsilon }(t+\theta )-u_{\varepsilon
}(t)\right\vert ^{p^{\prime }}dt\geq \nu _{n}M_{\delta }\right) .
\end{split}%
\end{equation*}%
Thanks to Tchebychev's inequality we have
\begin{equation*}
\begin{split}
\mathbb{P}(u_{\varepsilon }& \notin Z_{\delta })\leq \frac{1}{L_{\delta }}%
\mathbb{E}\int_{0}^{T}\left\vert Du_{\varepsilon }(s)\right\vert ^{p}ds+%
\frac{1}{K_{\delta }}\mathbb{E}\sup_{s\in \lbrack 0,T]}\left\vert
u_{\varepsilon }(s)\right\vert ^{2} \\
& +\sum \frac{1}{\nu _{n}M_{\delta }}\mathbb{E}\sup_{\left\vert \theta
\right\vert \leq \mu _{n}}\int_{0}^{T}\left\vert u_{\varepsilon }(t+\theta
)-u_{\varepsilon }(t)\right\vert ^{p^{\prime }}dt.
\end{split}%
\end{equation*}%
From Lemmata \ref{l0:1} and \ref{l0.2} it follows that
\begin{equation*}
\mathbb{P}(u_{\varepsilon }\notin Z_{\delta })\leq \frac{C}{L_{\delta }}+%
\frac{C}{K_{\delta }}+\frac{C}{M_{\delta }}\sum \frac{1}{\nu _{n}}(\mu
_{n})^{\frac{p^{\prime }}{p}}.
\end{equation*}%
By Choosing
\begin{equation*}
K_{\delta }=L_{\delta }=\frac{6C}{\delta }\text{ and }M_{\delta }=\frac{%
6C\left( \sum \frac{1}{\nu _{n}}(\mu _{n})^{\frac{p^{\prime }}{p}}\right) }{%
\delta },
\end{equation*}%
we have that
\begin{equation}
\mathbb{P}\left( u_{\varepsilon }\notin Z_{\delta }\right) \leq \frac{\delta
}{2}.  \label{0:17}
\end{equation}%
The sequence of probability measure $\Pi _{2}^{\varepsilon }=\mathbb{P}\circ
\psi _{2}^{\varepsilon }(A)=\mathbb{P}(W\in A)$ for any $A\in \mathcal{B}(%
\mathfrak{S}_{2})$ is constantly consisting of one element so it is weakly
compact. As $\mathcal{C}(0,T;\mathcal{U}_{0})$ is a Polish space we have
that a sequence of probability measures which is weakly compact is tight.
Therefore for any $\delta >0$ there exists a compact $\mathcal{K}_{\delta
}\subset \mathfrak{S}_{2}$ such that $\mathbb{P}(W\in \mathcal{K}_{\delta
})\geq 1-\frac{\delta }{2}$. It follows from this and \eqref{0:17} that
\begin{equation*}
\mathbb{P}\left( (u_{\varepsilon },W)\in Z_{\delta }\times \mathcal{K}%
_{\delta }\right) \geq 1-\delta .
\end{equation*}%
So we have found that for any $\delta >0$ there is a compact $Z_{\delta
}\times \mathcal{K}_{\delta }\subset \mathfrak{S}$ such that
\begin{equation*}
\Pi ^{\varepsilon }(Z_{\delta }\times \mathcal{K}_{\delta })\geq 1-\delta .
\end{equation*}%
This prove that the family $\Pi ^{\varepsilon }$ is tight on $\mathfrak{S}%
=L^{p}(0,T;H)\times \mathcal{C}(0,T;\mathcal{U}_{0})$.
\end{proof}

It follows from Lemma \ref{l0.4} and Prokhorov's theorem that there exists a
subsequence $\Pi ^{\varepsilon _{j}}$ of $\Pi ^{\varepsilon }$ converging
weakly (in the sense of measure) to a probability measure $\Pi $. It emerges
from Skorokhod's theorem that we can find a new probability space $(\bar{%
\Omega},\bar{\mathcal{F}},\bar{\mathbb{P}})$ and random variables $%
(u_{\varepsilon _{j}},W^{\varepsilon _{j}})$, $(u_{0},\bar{W})$ defined on
this new probability space and taking values in $\mathfrak{S}$ such that:
\begin{equation}
\text{The probability law of }({W}^{\varepsilon _{j}},u_{\varepsilon _{j}})%
\text{ is }\Pi ^{\varepsilon _{j}},  \label{0:18}
\end{equation}%
\begin{equation}
\text{The probability law of }(\bar{W},u_{0})\text{ is }\Pi ,\ \ \ \ \ \ \
\label{0:19}
\end{equation}%
\begin{equation}
{W}^{\varepsilon _{j}}\rightarrow \bar{W}\text{ in }\mathcal{C}(0,T;\mathcal{%
U}_{0})\,\,\bar{\mathbb{P}}\text{-a.s.,\ \ \ \ \ \ \ \ \ \ }  \label{0:20}
\end{equation}%
\begin{equation}
u_{\varepsilon _{j}}\rightarrow u_{0}\text{ in }L^{p}(0,T;H)\,\,\bar{\mathbb{%
P}}\text{-a.s..\ \ \ \ \ \ \ \ \ \ \ \ \ \ \ \ \ \ \ }  \label{0:21}
\end{equation}%
We can see that $\left\{ W^{\varepsilon _{j}}\right\} $ is a sequence of
cylindrical Brownian Motions evolving on $\mathcal{U}$. We let $\bar{%
\mathcal{F}}^{t}$ be the $\sigma $-algebra generated by $(\bar{W}%
(s),u_{0}(s)),0\leq s\leq t$ and the null sets of $\bar{\mathcal{F}}$. We
can show by arguing as in \cite{bensoussan2} that $\bar{W}$ is an $\bar{%
\mathcal{F}}^{t}$-adapted cylindrical Wiener process evolving in $\mathcal{U}
$. By the same argument as in \cite{bensoussan} we can show that for all $%
\phi \in V$ and for almost every $(\omega ,t)\in \bar{\Omega}\times \lbrack
0,T]$ the following holds true
\begin{equation}
\begin{split}
(u_{\varepsilon _{j}}(t),\phi )+\int_{0}^{t}a^{\varepsilon _{j}}(\cdot
,u_{\varepsilon _{j}}(s),Du_{\varepsilon _{j}})\cdot D\phi ds& =(u^{0},\phi
)-\int_{0}^{t}a_{0}^{\varepsilon _{j}}(\cdot ,u_{\varepsilon _{j}}(s))\phi ds
\\
& +\sum_{k=1}^{\infty }\int_{0}^{t}(M_{k}^{\varepsilon _{j}}(\cdot
,u_{\varepsilon _{j}}(s)),\phi )dW_{k}^{\varepsilon _{j}}.
\end{split}
\label{4.10}
\end{equation}

\section{Preliminary results}

In this section we collect some useful results that will be necessary in the
homogenization process. The notation is that of the preceding sections.
Before we can go further, let us however observe that property (\ref{3.1})
(in Definition \ref{d3.1}) still holds true for $f$ in $B(\Omega ;\mathcal{C}%
(\overline{Q}_{T};B_{A}^{p^{\prime },\infty }))$ where $B_{A}^{p^{\prime
},\infty }=B_{A}^{p^{\prime }}\cap L^{\infty }(\mathbb{R}_{y,\tau }^{N+1})$
and $p^{\prime }=p/(p-1)$.

With this in mind, the following assumption will be fundamental in the rest
of the paper:
\begin{equation}
\begin{array}{l}
a_{i}(x,t,\cdot ,\cdot ,\mu ,\lambda )\in B_{A}^{p^{\prime }}\text{ and }%
a_{0}(x,t,\cdot ,\cdot ,\mu ),\;M_{k}(\cdot ,\cdot ,\mu )\in B_{A}^{2} \\
\text{for any }(x,t)\in \overline{Q}_{T}\text{ and each }(\mu ,\lambda )\in
\mathbb{R}^{N+1},\ 1\leq i\leq N,\;k\geq 1%
\end{array}
\label{5.1}
\end{equation}%
where $p^{\prime }=p/(p-1)$ with $2\leq p<\infty $.

Arguing exactly as in \cite[Proposition 4.5]{NA} we have the following
result.

\begin{proposition}
\label{p5.1}Let $1\leq i\leq N$. Assume \emph{(\ref{5.1})} holds true. Then
for every $(\psi _{0},\Psi )\in (A)^{N+1}$ and every $(x,t)\in \overline{Q}%
_{T}$, the functions $(y,\tau )\mapsto a_{i}(x,t,y,\tau ,\psi _{0}(y,\tau
),\Psi (y,\tau ))$, $(y,\tau )\mapsto M_{k}(y,\tau ,\psi _{0}(y,\tau ))$ and
$(y,\tau )\mapsto a_{0}(x,t,y,\tau ,\psi _{0}(y,\tau ))$ denoted
respectively by $a_{i}(x,t,\cdot ,\cdot ,\psi _{0},\Psi )$, $M_{k}(\cdot
,\cdot ,\psi _{0})$ and $a_{0}(x,t,\cdot ,\cdot ,\psi _{0})$, lie
respectively in $B_{A}^{p^{\prime }}$, $B_{A}^{2}$ and $B_{A}^{2}$.
\end{proposition}

Now, let $(\psi _{0},\Psi )\in B(\Omega ;\mathcal{C}(\overline{Q}%
_{T};(A)^{N+1}))$. Assuming (\ref{5.1}), it can be easily shown that the
function $(x,t,y,\tau ,\omega )\mapsto a_{i}(x,t,y,\tau ,\psi
_{0}(x,t,y,\tau ,\omega ),\Psi (x,t,y,\tau ,\omega ))$, denoted by $%
a_{i}(\cdot ,\psi _{0},\Psi )$, lies in $B(\Omega ;\mathcal{C}(\overline{Q}%
_{T};B_{A}^{p^{\prime },\infty }))$ (use also Proposition \ref{p5.1}). We
can then define its trace
\begin{equation*}
(x,t,\omega )\mapsto a_{i}(x,t,x/\varepsilon ,t/\varepsilon ,\psi
_{0}(x,t,x/\varepsilon ,t/\varepsilon ,\omega ),\Psi (x,t,x/\varepsilon
,t/\varepsilon ,\omega )),
\end{equation*}%
from $Q_{T}\times \Omega $\ into $\mathbb{R}$, as an element of $L^{\infty
}(Q_{T}\times \Omega )$, which we denote by $a_{i}^{\varepsilon }(\cdot
,\psi _{0}^{\varepsilon },\Psi ^{\varepsilon })$. Likewise we can define the
functions $a_{0}(\cdot ,\psi _{0})$ and $a_{0}^{\varepsilon }(\cdot ,\psi
_{0}^{\varepsilon })$, $M_{k}(\cdot ,\psi _{0})$ and $M_{k}^{\varepsilon
}(\cdot ,\psi _{0}^{\varepsilon })$.

The next result allows us to rigorously set the homogenized problem.

\begin{proposition}
\label{p5.2}Let $2\leq p<\infty $ and let $1\leq i\leq N$. Assume \emph{(\ref%
{5.1})} holds. For any $(\psi _{0},\Psi )\in B(\Omega ;\mathcal{C}(\overline{%
Q}_{T};(A)^{N+1}))$ we have
\begin{equation}
a_{i}^{\varepsilon }(\cdot ,\psi _{0}^{\varepsilon },\Psi ^{\varepsilon
})\rightarrow a_{i}(\cdot ,\psi _{0},\Psi )\text{ in }L^{p^{\prime
}}(Q_{T}\times \Omega )\text{-weak }\Sigma \text{ as }\varepsilon
\rightarrow 0.  \label{5.2}
\end{equation}%
Let $a(\cdot ,\psi _{0},\Psi )=(a_{i}(\cdot ,\psi _{0},\Psi ))_{1\leq i\leq
N}$. The mapping $(\psi _{0},\Psi )\mapsto a(\cdot ,\psi _{0},\Psi )$ of $%
B(\Omega ;\mathcal{C}(\overline{Q}_{T};(A)^{N+1}))$ into $L^{p^{\prime
}}(Q_{T}\times \Omega ;B_{A}^{p^{\prime }})^{N}$ extends by continuity to a
unique mapping still denoted by $a$, of $L^{p}(Q_{T}\times \Omega
;(B_{A}^{p})^{N+1})$ into $L^{p^{\prime }}(Q_{T}\times \Omega
;B_{A}^{p^{\prime }})^{N}$ such that
\begin{equation*}
(a(\cdot ,u,\mathbf{v})-a(\cdot ,u,\mathbf{w}))\cdot (\mathbf{v}-\mathbf{w}%
)\geq c_{1}\left| \mathbf{v}-\mathbf{w}\right| ^{p}\text{\ a.e. in }%
Q_{T}\times \Omega \times \mathbb{R}_{y}^{N}\times \mathbb{R}_{\tau }
\end{equation*}%
\begin{equation*}
\left\| a_{i}(\cdot ,u,\mathbf{v})\right\| _{L^{p^{\prime }}(Q_{T}\times
\Omega ;B_{A}^{p^{\prime }})}\leq c_{2}^{\prime \prime }\left( 1+\left\|
u\right\| _{L^{p}(Q_{T}\times \Omega ;B_{A}^{p})}^{p-1}+\left\| \mathbf{v}%
\right\| _{L^{p}(Q_{T}\times \Omega ;(B_{A}^{p})^{N})}^{p-1}\right)
\end{equation*}%
\begin{equation}
\begin{array}{l}
\left\| a_{i}(\cdot ,u,\mathbf{v})-a_{i}(\cdot ,u,\mathbf{w})\right\|
_{L^{p^{\prime }}(Q_{T}\times \Omega ;B_{A}^{p^{\prime }})} \\
\;\;\;\;\;\leq c_{0}\left\| 1+\left| u\right| +\left| \mathbf{v}\right|
+\left| \mathbf{w}\right| \right\| _{L^{p}(Q_{T}\times \Omega
;B_{A}^{p})}^{p-2}\left\| \mathbf{v}-\mathbf{w}\right\| _{L^{p}(Q_{T}\times
\Omega ;(B_{A}^{p})^{N})}%
\end{array}
\label{5.3}
\end{equation}%
\begin{equation*}
\begin{array}{l}
\left| a_{i}(x,t,y,\tau ,u,\mathbf{w})-a_{i}(x^{\prime },t^{\prime },y,\tau
,v,\mathbf{w})\right| \leq \\
\;\;\leq m(\left| x-x^{\prime }\right| +\left| t-t^{\prime }\right| +\left|
u-v\right| )\left( 1+\left| u\right| ^{p-1}+\left| v\right| ^{p-1}+\left|
\mathbf{w}\right| ^{p-1}\right) \\
\text{a.e. in }Q_{T}\times \Omega \times \mathbb{R}_{y}^{N}\times \mathbb{R}%
_{\tau }%
\end{array}%
\end{equation*}%
for all $u,v\in L^{p}(Q_{T}\times \Omega ;B_{A}^{p})$, $\mathbf{v},\mathbf{w}%
\in L^{p}(Q_{T}\times \Omega ;(B_{A}^{p})^{N})$ and all $(x,t),(x^{\prime
},t^{\prime })\in Q_{T}$, where the constant $c_{2}^{\prime \prime }$
depends only on $c_{2}$ and on $Q_{T}$.
\end{proposition}

\begin{proof}
As discussed above before the statement of the proposition, we know that the
function $a_{i}(\cdot ,\psi _{0},\Psi )$ lies in $B(\Omega ;\mathcal{C}(%
\overline{Q}_{T};B_{A}^{p^{\prime },\infty }))$. Since Property (\ref{3.1})
(in Definition \ref{d3.1}) still holds for $f\in B(\Omega ;\mathcal{C}(%
\overline{Q}_{T};B_{A}^{p^{\prime },\infty }))$ the convergence result (\ref%
{5.2}) follows at once. Besides, it is immediate from the definition of the
function $a_{i}(\cdot ,\psi _{0},\Psi )$ (for $(\psi _{0},\Psi )\in B(\Omega
;\mathcal{C}(\overline{Q}_{T};(A)^{N+1}))$) and from some obvious arguments
that the remainder of the proposition follows from \cite[Proposition 3.1]%
{AMPA}.
\end{proof}

It emerges from the preceding proposition, the following important corollary.

\begin{corollary}
\label{c5.1}\emph{(1)} Let $(u_{\varepsilon })_{\varepsilon \in E}$ be a
sequence in $L^{2}(Q_{T}\times \Omega )$ such that $u_{\varepsilon
}\rightarrow u_{0}$ in $L^{2}(Q_{T}\times \Omega )$ as $E\ni \varepsilon
\rightarrow 0$, where $u_{0}\in L^{p}(Q_{T}\times \Omega )$. Let $\Psi \in
B(\Omega ;\mathcal{C}(\overline{Q}_{T};(A)^{N}))$, and finally let $1\leq
i\leq N$. Then, as $E\ni \varepsilon \rightarrow 0$,
\begin{equation*}
a_{i}^{\varepsilon }(\cdot ,u_{\varepsilon },\Psi ^{\varepsilon
})\rightarrow a_{i}(\cdot ,u_{0},\Psi )\text{\ in }L^{p^{\prime
}}(Q_{T}\times \Omega )\text{-weak }\Sigma \text{.}\;\;\;\;\;\;\;\;\;\;\;\;%
\;\;\;\;
\end{equation*}%
\emph{(2)} Let $\psi _{0}\in B(\Omega )\otimes \mathcal{C}_{0}^{\infty
}(Q_{T})$ and $\psi _{1}\in B(\Omega )\otimes \mathcal{C}_{0}^{\infty
}(Q_{T})\otimes A^{\infty }$. For $\varepsilon >0$, let
\begin{equation}
\Phi _{\varepsilon }=\psi _{0}+\varepsilon \psi _{1}^{\varepsilon
},\;\;\;\;\;\;\;\;\;\;\;\;\;\;  \label{5.4}
\end{equation}%
i.e., $\Phi _{\varepsilon }(x,t,\omega )=\psi _{0}(x,t,\omega )+\varepsilon
\psi _{1}(x,t,x/\varepsilon ,t/\varepsilon ,\omega )$\ for $(x,t,\omega )\in
Q_{T}\times \Omega $. Let $(u_{\varepsilon })_{\varepsilon \in E}$ be a
sequence in $L^{2}(Q_{T}\times \Omega )$ such that $u_{\varepsilon
}\rightarrow u_{0}$ in $L^{2}(Q_{T}\times \Omega )$ as $E\ni \varepsilon
\rightarrow 0$ where $u_{0}\in L^{2}(Q_{T}\times \Omega )$. Then, as $E\ni
\varepsilon \rightarrow 0$, one has
\begin{equation*}
(i)\;a_{i}^{\varepsilon }(\cdot ,u_{\varepsilon },D\Phi _{\varepsilon
})\rightarrow a_{i}(\cdot ,u_{0},D\psi _{0}+D_{y}\psi _{1})\text{ in }%
L^{p^{\prime }}(Q_{T}\times \Omega )\text{-weak }\Sigma \text{.}
\end{equation*}%
Moreover, if $(v_{\varepsilon })_{\varepsilon \in E}$ is a sequence in $%
L^{p}(Q_{T}\times \Omega )$ such that $v_{\varepsilon }\rightarrow v_{0}$ in
$L^{p}(Q_{T}\times \Omega )$-weak $\Sigma $ as $E\ni \varepsilon \rightarrow
0$ where $v_{0}\in L^{p}(Q_{T}\times \Omega ;\mathcal{B}_{A}^{p})$, then, as
$E\ni \varepsilon \rightarrow 0$,
\begin{equation*}
(ii)\;\int_{Q_{T}\times \Omega }a_{i}^{\varepsilon }(\cdot ,u_{\varepsilon
},D\Phi _{\varepsilon })v_{\varepsilon }dxdtd\mathbb{P}\rightarrow
\iint_{Q_{T}\times \Omega \times \Delta (A)}\widehat{a}_{i}(\cdot
,u_{0},D\psi _{0}+\partial \widehat{\psi }_{1})\widehat{v}_{0}dxdtd\mathbb{P}%
d\beta \text{.}
\end{equation*}
\end{corollary}

\begin{proof}
We just sketch the proof since it is very similar to the one of \cite[%
Corollaries 4.7-4.8]{NA}. For part (1), let $f\in L^{p}(Q_{T}\times \Omega
;A)$, and let $(\psi _{j})_{j}$ be a sequence in $B(\Omega )\otimes \mathcal{%
C}_{0}^{\infty }(Q_{T})$ such that $\psi _{j}\rightarrow u_{0}$ in $%
L^{p}(Q_{T}\times \Omega )$ as $j\rightarrow \infty $. We have
\begin{equation*}
\begin{array}{l}
\int_{Q_{T}\times \Omega }a_{i}^{\varepsilon }(\cdot ,u_{\varepsilon },\Psi
^{\varepsilon })f^{\varepsilon }dxdtd\mathbb{P}-\iint_{Q_{T}\times \Omega
\times \Delta (A)}\widehat{a}_{i}(\cdot ,u_{0},\widehat{\Psi })\widehat{f}%
dxdtd\mathbb{P}d\beta \\
\ \ =\int_{Q_{T}\times \Omega }\left[ a_{i}^{\varepsilon }(\cdot
,u_{\varepsilon },\Psi ^{\varepsilon })-a_{i}^{\varepsilon }(\cdot
,u_{0},\Psi ^{\varepsilon })\right] f^{\varepsilon }dxdtd\mathbb{P} \\
\;\;\;\;+\int_{Q_{T}\times \Omega }\left[ a_{i}^{\varepsilon }(\cdot
,u_{0},\Psi ^{\varepsilon })-a_{i}^{\varepsilon }(\cdot ,\psi _{j},\Psi
^{\varepsilon })\right] f^{\varepsilon }dxdtd\mathbb{P} \\
\;\;\;\;\;\;+\int_{Q_{T}\times \Omega }a_{i}^{\varepsilon }(\cdot ,\psi
_{j},\Psi ^{\varepsilon })f^{\varepsilon }dxdt-\iint_{Q_{T}\times \Omega
\times \Delta (A)}\widehat{a}_{i}(\cdot ,u_{0},\widehat{\Psi })\widehat{f}%
dxdtd\mathbb{P}d\beta \\
\;\;=A_{\varepsilon }+B_{\varepsilon ,j}+C_{\varepsilon ,j}%
\end{array}%
\end{equation*}%
where
\begin{equation*}
\begin{array}{l}
A_{\varepsilon }=\int_{Q_{T}\times \Omega }\left[ a_{i}^{\varepsilon }(\cdot
,u_{\varepsilon },\Psi ^{\varepsilon })-a_{i}^{\varepsilon }(\cdot
,u_{0},\Psi ^{\varepsilon })\right] f^{\varepsilon }dxdtd\mathbb{P}, \\
B_{\varepsilon ,j}=\int_{Q_{T}\times \Omega }\left[ a_{i}^{\varepsilon
}(\cdot ,u_{0},\Psi ^{\varepsilon })-a_{i}^{\varepsilon }(\cdot ,\psi
_{j},\Psi ^{\varepsilon })\right] f^{\varepsilon }dxdtd\mathbb{P}, \\
C_{\varepsilon ,j}=\int_{Q_{T}\times \Omega }a_{i}^{\varepsilon }(\cdot
,\psi _{j},\Psi ^{\varepsilon })f^{\varepsilon }dxdt-\iint_{Q_{T}\times
\Omega \times \Delta (A)}\widehat{a}_{i}(\cdot ,u_{0},\widehat{\Psi })%
\widehat{f}dxdtd\mathbb{P}d\beta .%
\end{array}%
\end{equation*}%
As far as $A_{\varepsilon }$ is concerned, we have
\begin{equation*}
\left| A_{\varepsilon }\right| \leq \int_{Q_{T}\times \Omega }m\left( \left|
u_{\varepsilon }-u_{0}\right| \right) \left( 1+\left| u_{\varepsilon
}\right| ^{p-1}+\left| u_{0}\right| ^{p-1}+\left| \Psi ^{\varepsilon
}\right| ^{p-1}\right) \left| f^{\varepsilon }\right| dxdtd\mathbb{P}.
\end{equation*}%
From the convergence result $u_{\varepsilon }\rightarrow u_{0}$ in $%
L^{2}(Q_{T}\times \Omega )$, we infer that $m\left( \left| u_{\varepsilon
}-u_{0}\right| \right) \rightarrow 0$ a.e. in $Q_{T}\times \Omega $ as $E\ni
\varepsilon \rightarrow 0$, so that, by Egorov's theorem, $A_{\varepsilon
}\rightarrow 0$ as $E\ni \varepsilon \rightarrow 0$. As for $C_{\varepsilon
,j}$, we see that the function $(x,t,\omega )\mapsto a_{i}(x,t,\cdot ,\cdot
,\psi _{j}(x,t,\omega ),\Psi (x,t,\cdot ,\cdot ,\omega ))$ belongs to $%
B(\Omega ;\mathcal{C}(\overline{Q}_{T};B_{A}^{p^{\prime },\infty }))$, in
such a way that we use the convergence result (\ref{5.2}) to get
\begin{equation*}
C_{\varepsilon ,j}\rightarrow \iint_{Q_{T}\times \Omega \times \Delta
(A)}\left( \widehat{a}_{i}(\cdot ,\psi _{j},\widehat{\Psi })-\widehat{a}%
_{i}(\cdot ,u_{0},\widehat{\Psi })\right) \widehat{f}dxdtd\mathbb{P}d\beta
\equiv \widehat{C}_{j}\text{\ as }E\ni \varepsilon \rightarrow 0.
\end{equation*}%
But as
\begin{equation*}
\left| \widehat{C}_{j}\right| \leq \iint_{Q_{T}\times \Omega \times \Delta
(A)}m\left( \left| \psi _{j}-u_{0}\right| \right) \left( 1+\left| \psi
_{j}\right| ^{p-1}+\left| u_{0}\right| ^{p-1}+\left| \widehat{\Psi }\right|
^{p-1}\right) \left| \widehat{f}\right| dxdtd\mathbb{P}d\beta ,
\end{equation*}%
arguing as before we get $\widehat{C}_{j}\rightarrow 0$ as $j\rightarrow
\infty $. We also have $\lim_{E\ni \varepsilon \rightarrow
0}\lim_{j\rightarrow \infty }B_{\varepsilon ,j}=0$, so that part (1) follows
from the equality
\begin{equation*}
\begin{array}{l}
\lim_{E\ni \varepsilon \rightarrow 0}\left( \int_{Q_{T}\times \Omega
}a_{i}^{\varepsilon }(\cdot ,u_{\varepsilon },\Psi ^{\varepsilon
})f^{\varepsilon }dxdtd\mathbb{P}-\iint_{Q_{T}\times \Omega \times \Delta
(A)}\widehat{a}_{i}(\cdot ,u_{0},\widehat{\Psi })\widehat{f}dxdtd\mathbb{P}%
d\beta \right) \\
\;=\lim_{E\ni \varepsilon \rightarrow 0}A_{\varepsilon }+\lim_{E\ni
\varepsilon \rightarrow 0}\lim_{j\rightarrow \infty }B_{\varepsilon
,j}+\lim_{E\ni \varepsilon \rightarrow 0}\lim_{j\rightarrow \infty
}C_{\varepsilon ,j}=0.%
\end{array}%
\end{equation*}%
Part (2) is a mere consequence of part (1).
\end{proof}

Another important result which will be needed is the

\begin{lemma}
\label{l5.1}Let $(u_{\varepsilon })_{\varepsilon }$ be a sequence in $%
L^{2}(Q_{T}\times \Omega )$ such that $u_{\varepsilon }\rightarrow u_{0}$ in
$L^{2}(Q_{T}\times \Omega )$ as $\varepsilon \rightarrow 0$ where $u_{0}\in
L^{2}(Q_{T}\times \Omega )$. Then for each positive integer $k$ we have ,
\begin{equation*}
M_{k}^{\varepsilon }(\cdot ,u_{\varepsilon })\rightarrow M_{k}(\cdot ,u_{0})%
\text{ in }L^{2}(Q_{T}\times \Omega )\text{-weak }\Sigma \text{ as }%
\varepsilon \rightarrow 0.
\end{equation*}
\end{lemma}

\begin{proof}
First of all, let $u\in B(\Omega ;\mathcal{C}(\overline{Q}_{T}))$; then the
function $(x,t,y,\tau ,\omega )\mapsto M_{k}(y,\tau ,u(x,t,\omega ))$ lies
in $B(\Omega ;\mathcal{C}(\overline{Q}_{T};B_{A}^{2,\infty }))$, so that we
have $M_{k}^{\varepsilon }(\cdot ,u)\rightarrow M_{k}(\cdot ,u)$ in $%
L^{2}(Q_{T}\times \Omega )$-weak $\Sigma $ as $\varepsilon \rightarrow 0$.
Next, since $B(\Omega ;\mathcal{C}(\overline{Q}_{T}))$ is dense in $%
L^{2}(Q_{T}\times \Omega )$, it can be easily shown that
\begin{equation}
M_{k}^{\varepsilon }(\cdot ,u_{0})\rightarrow M_{k}(\cdot ,u_{0})\text{ in }%
L^{2}(Q_{T}\times \Omega )\text{-weak }\Sigma \text{ as }\varepsilon
\rightarrow 0.  \label{Eq1}
\end{equation}%
Now, let $f\in L^{2}(\Omega ;L^{2}(Q_{T};A))$; then
\begin{eqnarray*}
&&\int_{Q_{T}\times \Omega }M_{k}^{\varepsilon }(\cdot ,u_{\varepsilon
})f^{\varepsilon }dxdtd\mathbb{P}-\iint_{Q_{T}\times \Omega \times \Delta
(A)}\widehat{M}_{k}(\cdot ,u_{0})\widehat{f}dxdtd\mathbb{P}d\beta \\
&=&\int_{Q_{T}\times \Omega }(M_{k}^{\varepsilon }(\cdot ,u_{\varepsilon
})-M_{k}^{\varepsilon }(\cdot ,u_{0}))f^{\varepsilon }dxdtd\mathbb{P}%
+\int_{Q_{T}\times \Omega }M_{k}^{\varepsilon }(\cdot ,u_{0})f^{\varepsilon
}dxdtd\mathbb{P} \\
&&-\iint_{Q_{T}\times \Omega \times \Delta (A)}\widehat{M}_{k}(\cdot ,u_{0})%
\widehat{f}dxdtd\mathbb{P}d\beta .
\end{eqnarray*}%
Using the inequality
\begin{equation*}
\left| \int_{Q_{T}\times \Omega }(M_{k}^{\varepsilon }(\cdot ,u_{\varepsilon
})-M_{k}^{\varepsilon }(\cdot ,u_{0}))f^{\varepsilon }dxdtd\mathbb{P}\right|
\leq C\left\| u_{\varepsilon }-u_{0}\right\| _{L^{2}(Q_{T}\times \Omega
)}\left\| f^{\varepsilon }\right\| _{L^{2}(Q_{T}\times \Omega )}
\end{equation*}%
in conjunction with (\ref{Eq1}) leads at once to the result.
\end{proof}

\begin{remark}
\label{r5.1}\emph{From the Lipschitz property of the function }$M_{k}$\emph{%
\ we may get more information on the limit of the sequence }$%
M_{k}^{\varepsilon }(\cdot ,u_{\varepsilon })$\emph{. Indeed, since }$\left|
M_{k}^{\varepsilon }(\cdot ,u_{\varepsilon })-M_{k}^{\varepsilon }(\cdot
,u_{0})\right| \leq C\left| u_{\varepsilon }-u_{0}\right| $\emph{, we deduce
the following convergence result: }%
\begin{equation*}
M_{k}^{\varepsilon }(\cdot ,u_{\varepsilon })\rightarrow \widetilde{M}%
_{k}(u_{0})\text{\ in }L^{2}(Q_{T}\times \Omega )\text{ as }\varepsilon
\rightarrow 0
\end{equation*}%
\emph{where }$\widetilde{M}_{k}(u_{0})(x,t,\omega )=\int_{\Delta (A)}%
\widehat{M}_{k}(s,s_{0},u_{0}(x,t,\omega ))d\beta $\emph{, so that we can
derive the existence of a subsequence of }$M_{k}^{\varepsilon }(\cdot
,u_{\varepsilon })$\emph{\ that converges a.e. in }$Q_{T}\times \Omega $%
\emph{\ to }$\widetilde{M}_{k}(u_{0})$\emph{. For the next sections we will
need the following function: }$\widetilde{M}(u_{0})=(\widetilde{M}%
_{k}(u_{0}))_{k\geq 1}$\emph{.}
\end{remark}

We end this section with some useful spaces. Let
\begin{equation*}
\mathbb{F}_{0}^{1,p}=L^{p}(\bar{\Omega}\times \left( 0,T\right)
;W_{0}^{1,p}(Q))\times L^{p}(Q_{T}\times \bar{\Omega};\mathcal{V})
\end{equation*}%
and
\begin{equation*}
\mathcal{F}_{0}^{\infty }=[B(\bar{\Omega})\otimes \mathcal{C}_{0}^{\infty
}(Q_{T})]\times \lbrack B(\bar{\Omega})\otimes \mathcal{C}_{0}^{\infty
}(Q_{T})\otimes \mathcal{E}]
\end{equation*}%
where $\mathcal{V}=\mathcal{B}_{A_{\tau }}^{p}(\mathbb{R}_{\tau };\mathcal{B}%
_{\#A_{y}}^{1,p})$ and $\mathcal{E}=\varrho _{\tau }(A_{\tau }^{\infty
})\otimes \lbrack \varrho _{y}(A_{y}^{\infty }/\mathbb{R})]$, and $\varrho
_{\tau }$ (resp. $\varrho _{y}$) denotes the canonical mapping of $%
B_{A_{\tau }}^{p}$ (resp. $B_{A_{y}}^{p}$) onto $\mathcal{B}_{A_{\tau }}^{p}$
(resp. $\mathcal{B}_{A_{y}}^{p}$). $\mathbb{F}_{0}^{1,p}$ is a Banach space
under the norm
\begin{equation*}
\left\| (u_{0},u_{1})\right\| _{\mathbb{F}_{0}^{1,p}}=\left\| u_{0}\right\|
_{L^{p}(\bar{\Omega}\times \left( 0,T\right) ;W_{0}^{1,p}(Q))}+\left\|
u_{1}\right\| _{L^{p}(Q_{T}\times \bar{\Omega};\mathcal{V})}.
\end{equation*}%
Moreover, since $B(\bar{\Omega})$ is dense in $L^{p}(\bar{\Omega})$, it is
an easy matter to check that $\mathcal{F}_{0}^{\infty }$ is dense in $%
\mathbb{F}_{0}^{1,p}$.

\section{Homogenization results}

Let $(u_{\varepsilon _{j}})$ be the sequence determined in Section 4 and
satisfying Eq. (\ref{4.10}). It therefore satisfies the a priori estimates (%
\ref{0:3})-(\ref{0:4}), so that, by the diagonal process, one can find a
subsequence of $(u_{\varepsilon _{j}})_{j}$ not relabeled, which weakly
converges in $L^{p}(\bar{\Omega};L^{p}(0,T;W_{0}^{1,p}(Q)))$ to $u_{0}$
determined by the Skorokhod's theorem and satisfying (\ref{0:21}). Next, due
to the estimate (\ref{0:3}) (which yields the uniform integrability of the
sequence $(u_{\varepsilon _{j}})_{j}$ with respect to $\omega $) and the
Vitali's theorem, we deduce from (\ref{0:21}) that, as $j\rightarrow \infty $%
,
\begin{equation}
u_{\varepsilon _{j}}\rightarrow u_{0}\text{ in }L^{2}(Q_{T}\times \bar{\Omega%
}).  \label{6.1}
\end{equation}%
Then, from Theorem \ref{t3.3}, we infer the existence of a function $%
u_{1}\in L^{p}(\bar{\Omega};L^{p}(Q_{T};\mathcal{B}_{A_{\tau }}^{p}(\mathbb{R%
}_{\tau };\mathcal{B}_{\#A_{y}}^{1,p})))$ such that
\begin{equation}
\frac{\partial u_{\varepsilon _{j}}}{\partial x_{i}}\rightarrow \frac{%
\partial u_{0}}{\partial x_{i}}+\frac{\overline{\partial }u_{1}}{\partial
y_{i}}\text{ in }L^{p}(Q_{T}\times \bar{\Omega})\text{-weak }\Sigma \text{ }%
(1\leq i\leq N)  \label{6.2}
\end{equation}%
hold when $\varepsilon _{j}\rightarrow 0$.

With this in mind, the following \textit{global} homogenization result holds.

\begin{proposition}
\label{p6.1}The couple $(u_{0},u_{1})\in \mathbb{F}_{0}^{1,p}$ determined
above solves the following variational problem
\begin{equation}
\left\{
\begin{array}{l}
-\int_{Q_{T}\times \bar{\Omega}}u_{0}\psi _{0}^{\prime }dxdtd\bar{\mathbb{P}}%
+\iint_{Q_{T}\times \bar{\Omega}\times \Delta (A)}\widehat{a}_{0}(\cdot
,u_{0})\psi _{0}dxdtd\bar{\mathbb{P}}d\beta \\
+\iint_{Q_{T}\times \bar{\Omega}\times \Delta (A)}\widehat{a}(\cdot
,u_{0},Du_{0}+\partial \widehat{u}_{1})\cdot (D\psi _{0}+\partial \widehat{%
\psi }_{1})dxdtd\bar{\mathbb{P}}d\beta \\
=\int_{\bar{\Omega}}\int_{0}^{T}\left( \widetilde{M}(u_{0}),\psi _{0}\right)
d\bar{W}d\bar{\mathbb{P}}\text{ for all }(\psi _{0},\psi _{1})\in \mathcal{F}%
_{0}^{\infty }\text{.}%
\end{array}%
\right.  \label{6.3}
\end{equation}
\end{proposition}

\begin{proof}
In what follows, we omit the index $j$ from the sequence $\varepsilon _{j}$.
So we will merely write $\varepsilon $ for $\varepsilon _{j}$. With this in
mind, let $\Phi =(\psi _{0},\varrho \circ \psi _{1})\in \mathcal{F}%
_{0}^{\infty }$ with $\psi _{0}\in B(\bar{\Omega})\otimes \mathcal{C}%
_{0}^{\infty }(Q_{T})$, $\psi _{1}\in B(\bar{\Omega})\otimes \mathcal{C}%
_{0}^{\infty }(Q_{T})\otimes \lbrack A_{\tau }^{\infty }\otimes
(A_{y}^{\infty }/\mathbb{R})]$. Define $\Phi _{\varepsilon }$ as in (\ref%
{5.4}) (see Corollary \ref{c5.1}). Then, $\Phi _{\varepsilon }\in B(\bar{%
\Omega})\otimes \mathcal{C}_{0}^{\infty }(Q_{T})$ and, using $\Phi
_{\varepsilon }$ as a test function in the variational formulation of (\ref%
{4.10}) we get
\begin{eqnarray*}
\int_{\bar{\Omega}}\left( u_{\varepsilon }(T),\Phi _{\varepsilon }(T)\right)
d\bar{\mathbb{P}} &=&\int_{\bar{\Omega}}\left( u^{0},\Phi _{\varepsilon
}(0)\right) d\bar{\mathbb{P}}+\int_{Q_{T}\times \bar{\Omega}}u_{\varepsilon }%
\frac{\partial \Phi _{\varepsilon }}{\partial t}dxdtd\bar{\mathbb{P}} \\
&&-\int_{Q_{T}\times \bar{\Omega}}a^{\varepsilon }(\cdot ,u_{\varepsilon
},Du_{\varepsilon })\cdot D\Phi _{\varepsilon }dxdtd\bar{\mathbb{P}} \\
&&-\int_{Q_{T}\times \bar{\Omega}}a_{0}^{\varepsilon }(\cdot ,u_{\varepsilon
})\Phi _{\varepsilon }dxdtd\bar{\mathbb{P}} \\
&&+\int_{\bar{\Omega}}\int_{0}^{T}\left( M^{\varepsilon }(\cdot
,u_{\varepsilon }),\Phi _{\varepsilon }\right) dW^{\varepsilon }d\bar{%
\mathbb{P}},
\end{eqnarray*}%
or equivalently, taking into account the fact that $\Phi _{\varepsilon
}(0)=\Phi _{\varepsilon }(T)=0$,
\begin{eqnarray}
&&-\int_{Q_{T}\times \bar{\Omega}}u_{\varepsilon }\frac{\partial \Phi
_{\varepsilon }}{\partial t}dxdtd\bar{\mathbb{P}}+\int_{Q_{T}\times \bar{%
\Omega}}a^{\varepsilon }(\cdot ,u_{\varepsilon },Du_{\varepsilon })\cdot
D\Phi _{\varepsilon }dxdtd\bar{\mathbb{P}}  \label{6.4} \\
&&+\int_{Q_{T}\times \bar{\Omega}}a_{0}^{\varepsilon }(\cdot ,u_{\varepsilon
})\Phi _{\varepsilon }dxdtd\bar{\mathbb{P}}=\int_{0}^{T}\int_{\bar{\Omega}%
}\left( M^{\varepsilon }(\cdot ,u_{\varepsilon }),\Phi _{\varepsilon
}\right) dW^{\varepsilon }d\bar{\mathbb{P}}.  \notag
\end{eqnarray}%
We consider the terms in (\ref{6.4}) respectively.

Firstly we have
\begin{eqnarray*}
\int_{Q_{T}\times \bar{\Omega}}u_{\varepsilon }\frac{\partial \Phi
_{\varepsilon }}{\partial t}dxdtd\bar{\mathbb{P}} &=&\int_{Q_{T}\times \bar{%
\Omega}}u_{\varepsilon }\frac{\partial \psi _{0}}{\partial t}dxdtd\bar{%
\mathbb{P}}+\varepsilon \int_{Q_{T}\times \bar{\Omega}}u_{\varepsilon
}\left( \frac{\partial \psi _{1}}{\partial t}\right) ^{\varepsilon }dxdtd%
\bar{\mathbb{P}} \\
&&+\int_{Q_{T}\times \bar{\Omega}}u_{\varepsilon }\left( \frac{\partial \psi
_{1}}{\partial \tau }\right) ^{\varepsilon }dxdtd\bar{\mathbb{P}}.
\end{eqnarray*}%
But in view of (\ref{0:21}) we have that $u_{\varepsilon }\rightarrow u_{0}$
in $L^{2}(Q_{T}\times \bar{\Omega})$ (strong). Moreover, since $(\partial
\psi _{1}/\partial \tau )^{\varepsilon }\rightarrow M(\partial \psi
_{1}/\partial \tau )=0$ in $L^{2}(Q_{T}\times \bar{\Omega})$-weak, we deduce
from the preceding strong convergence result that
\begin{equation*}
\int_{Q_{T}\times \bar{\Omega}}u_{\varepsilon }\frac{\partial \Phi
_{\varepsilon }}{\partial t}dxdtd\bar{\mathbb{P}}\rightarrow
\int_{Q_{T}\times \bar{\Omega}}u_{0}\frac{\partial \psi _{0}}{\partial t}%
dxdtd\bar{\mathbb{P}}.
\end{equation*}%
Next, from Corollary \ref{c5.1}, it follows that $a_{0}^{\varepsilon }(\cdot
,u_{\varepsilon })\rightarrow a_{0}(\cdot ,u_{0})$ in $L^{2}(Q_{T}\times
\bar{\Omega})$-weak $\Sigma $, so that
\begin{equation*}
\int_{Q_{T}\times \bar{\Omega}}a_{0}^{\varepsilon }(\cdot ,u_{\varepsilon
})\Phi _{\varepsilon }dxdtd\bar{\mathbb{P}}\rightarrow \iint_{Q_{T}\times
\bar{\Omega}\times \Delta (A)}\widehat{a}_{0}(\cdot ,u_{0})\psi _{0}dxdtd%
\bar{\mathbb{P}}d\beta .
\end{equation*}%
As far as the term $\int_{0}^{T}\int_{\bar{\Omega}}\left( M^{\varepsilon
}(\cdot ,u_{\varepsilon }),\Phi _{\varepsilon }\right) dW^{\varepsilon }d%
\bar{\mathbb{P}}$ is concerned, thanks to Remark \ref{r5.1} we get at once
\begin{equation*}
\int_{\bar{\Omega}}\int_{0}^{T}\left( M^{\varepsilon }(\cdot ,u_{\varepsilon
}),\Phi _{\varepsilon }\right) dW^{\varepsilon }d\bar{\mathbb{P}}\rightarrow
\int_{\bar{\Omega}}\int_{0}^{T}\left( \widetilde{M}(u_{0}),\psi _{0}\right) d%
\bar{W}d\bar{\mathbb{P}}.
\end{equation*}%
The last term is more involved. Indeed, by the monotonicity argument, it
emerges that
\begin{equation}
\int_{Q_{T}\times \bar{\Omega}}(a^{\varepsilon }(\cdot ,u_{\varepsilon
},Du_{\varepsilon })-a^{\varepsilon }(\cdot ,u_{\varepsilon },D\Phi
_{\varepsilon }))\cdot (Du_{\varepsilon }-D\Phi _{\varepsilon })dxdtd\bar{%
\mathbb{P}}\geq 0\text{.}  \label{6.5}
\end{equation}%
Owing to the estimate (\ref{0:4}) (denoting by $\overline{\mathbb{E}}$ the
mathematical expectation on $(\bar{\Omega},\bar{\mathcal{F}},\bar{\mathbb{P}}%
)$) we infer that
\begin{equation*}
\sup_{\varepsilon >0}\overline{\mathbb{E}}\left\Vert a^{\varepsilon }(\cdot
,u_{\varepsilon },Du_{\varepsilon })\right\Vert _{L^{p^{\prime
}}(Q_{T})^{N}}^{p^{\prime }}<\infty ,
\end{equation*}%
so that, from Theorem \ref{t3.1}, there exist a function $\chi \in
L^{p^{\prime }}(Q_{T}\times \bar{\Omega};\mathcal{B}_{A}^{p^{\prime }})^{N}$
and a subsequence of $\varepsilon $ not relabeled, such that $a^{\varepsilon
}(\cdot ,u_{\varepsilon },Du_{\varepsilon })\rightarrow \chi $ in $%
L^{p^{\prime }}(Q_{T}\times \bar{\Omega})^{N}$-weak $\Sigma $ as $%
\varepsilon \rightarrow 0$. We therefore pass to the limit in (\ref{6.5})
(as $\varepsilon \rightarrow 0$) using Corollary \ref{c5.1} to get
\begin{equation}
\iint_{Q_{T}\times \bar{\Omega}\times \Delta (A)}(\widehat{\chi }-\widehat{a}%
(\cdot ,u_{0},\mathbb{D}\Phi ))\cdot (\mathbb{D}\mathbf{u}-\mathbb{D}\Phi
)dxdtd\bar{\mathbb{P}}d\beta \geq 0  \label{6.6}
\end{equation}%
for any $\Phi \in \mathcal{F}_{0}^{\infty }$ where $\mathbb{D}\mathbf{u}%
=Du_{0}+\partial \widehat{u}_{1}$ ($\mathbf{u}=(u_{0},u_{1})$) and $\mathbb{D%
}\Phi =D\psi _{0}+\partial \widehat{\psi }_{1}$. By a density and continuity
arguments (\ref{6.6}) still holds for $\Phi \in \mathbb{F}_{0}^{1,p}$. Hence
by taking $\Phi =\mathbf{u}+\lambda \mathbf{v}$ for $\mathbf{v}%
=(v_{0},v_{1})\in \mathbb{F}_{0}^{1,p}$ and $\lambda >0$ arbitrarily fixed,
we get
\begin{equation*}
\lambda \iint_{Q_{T}\times \bar{\Omega}\times \Delta (A)}(\widehat{\chi }-%
\widehat{a}(\cdot ,u_{0},\mathbb{D}\mathbf{u}+\lambda \mathbb{D}\mathbf{v}%
))\cdot \mathbb{D}\mathbf{v}dxdtd\bar{\mathbb{P}}d\beta \geq 0\;\forall
\mathbf{v}\in \mathbb{F}_{0}^{1,p}.
\end{equation*}%
Therefore by a mere routine, we deduce that $\chi =a(\cdot ,u_{0},Du_{0}+%
\overline{D}_{y}u_{1})$. Putting all the above facts together we are led to (%
\ref{6.3}), and the proof is completed.
\end{proof}

The problem (\ref{6.3}) is equivalent to the following system:
\begin{equation}
\iint_{Q_{T}\times \bar{\Omega}\times \Delta (A)}\widehat{a}(\cdot ,u_{0},%
\mathbb{D}\mathbf{u})\cdot \partial \widehat{\psi }_{1}dxdtd\bar{\mathbb{P}}%
d\beta =0\text{ for all }\psi _{1}\in B(\bar{\Omega})\otimes \mathcal{C}%
_{0}^{\infty }(Q_{T})\otimes \mathcal{E}  \label{6.7}
\end{equation}%
and
\begin{equation}
\left\{
\begin{array}{l}
-\int_{Q_{T}\times \bar{\Omega}}u_{0}\psi _{0}^{\prime }dxdtd\mathbb{\bar{P}}%
+\iint_{Q_{T}\times \bar{\Omega}\times \Delta (A)}\widehat{a}(\cdot ,u_{0},%
\mathbb{D}\mathbf{u})\cdot D\psi _{0}dxdtd\bar{\mathbb{P}}d\beta \\
+\iint_{Q_{T}\times \bar{\Omega}\times \Delta (A)}\widehat{a}_{0}(\cdot
,u_{0})\psi _{0}dxdtd\bar{\mathbb{P}}d\beta =\int_{\bar{\Omega}}\int_{0}^{T}(%
\widetilde{M}(u_{0}),\psi _{0})d\bar{W}d\bar{\mathbb{P}} \\
\text{for all }\psi _{0}\in B(\bar{\Omega})\otimes \mathcal{C}_{0}^{\infty
}(Q_{T}).%
\end{array}%
\right.  \label{6.8}
\end{equation}%
As far as (\ref{6.7}) is concerned, let $(x,t)\in Q_{T}$ and let $(r,\xi
)\in \mathbb{R}\times \mathbb{R}^{N}$ be freely fixed. Let $\pi (x,t,r,\xi )$
be defined by the cell problem
\begin{equation}
\begin{array}{l}
\pi (x,t,r,\xi )\in \mathcal{V}=\mathcal{B}_{A_{\tau }}^{p}(\mathbb{R}_{\tau
};\mathcal{B}_{\#A_{y}}^{1,p}): \\
\int_{\Delta (A)}\widehat{a}(\cdot ,r,\xi +\partial \widehat{\pi }(x,t,r,\xi
))\cdot \partial \widehat{w}d\beta =0\text{ for all }w\in \mathcal{V}.%
\end{array}
\label{6.9}
\end{equation}%
Then from the properties of the function $a$, it follows by \cite[Chap. 2]%
{Lions} that Eq. (\ref{6.9}) admits at least a solution. Now if $\pi
_{1}\equiv \pi _{1}(x,t,r,\xi )$ and $\pi _{2}\equiv \pi _{2}(x,t,r,\xi )$
are two solutions of (\ref{6.9}), then we must have
\begin{equation*}
\int_{\Delta (A)}\left( \widehat{a}(\cdot ,r,\xi +\partial \widehat{\pi }%
_{1})-\widehat{a}(\cdot ,r,\xi +\partial \widehat{\pi }_{2})\right) \cdot
\left( \partial \widehat{\pi }_{1}-\partial \widehat{\pi }_{2}\right) d\beta
=0,
\end{equation*}%
and so, by assumption A2.,$\;\partial \widehat{\pi }_{1}=\partial \widehat{%
\pi }_{2}$, so that $\frac{\overline{\partial }\pi _{1}}{\partial y_{i}}=%
\frac{\overline{\partial }\pi _{2}}{\partial y_{i}}$ ($1\leq i\leq N$) since
$\partial _{i}\widehat{\pi }_{j}=\mathcal{G}_{1}\left( \frac{\overline{%
\partial }\pi _{j}}{\partial y_{i}}\right) $ for $j=1,2$. Hence $\pi
_{1}=\pi _{2}$ since they belong to $\mathcal{V}$. Next, taking in
particular $r=u_{0}(x,t,\omega )$ and $\xi =Du_{0}(x,t,\omega )$ with $%
(x,t,\omega )$ arbitrarily chosen in $Q_{T}\times \bar{\Omega}$, and then
choosing in (\ref{6.7}) the particular test functions $\psi _{1}(x,t,\omega
)=\phi (\omega )\varphi (x,t)w$ ($(x,t,\omega )\in Q_{T}\times \bar{\Omega}$%
) with $\varphi \in \mathcal{C}_{0}^{\infty }(Q_{T})$, $\phi \in B(\bar{%
\Omega})$ and $w\in \mathcal{E}$, and finally comparing the resulting
equation with (\ref{6.9}) (note that $\mathcal{E}$ is dense in $\mathcal{V}$%
), the uniqueness of the solution to (\ref{6.7}) tells us that $u_{1}=\pi
(\cdot ,u_{0},Du_{0})$, where the right-hand side of the preceding equality
stands for the function $(x,t,\omega )\mapsto \pi (x,t,u_{0}(x,t,\omega
),Du_{0}(x,t,\omega ))$ from $Q_{T}\times \bar{\Omega}$ into $\mathcal{V}$.

We have just proved the

\begin{proposition}
\label{p6.2}The solution of the variational problem \emph{(\ref{6.7})} is
unique.
\end{proposition}

Let us now deal with the variational problem (\ref{6.8}). For that, set
\begin{equation*}
q(x,t,r,\xi )=\int_{\Delta (A)}\widehat{a}(\cdot ,r,\xi +\partial \widehat{%
\pi }(x,t,r,\xi ))d\beta \text{ }
\end{equation*}%
and
\begin{equation*}
q_{0}(x,t,r)=\int_{\Delta (A)}\widehat{a}_{0}(\cdot ,r)d\beta ;\;\widetilde{M%
}(r)=\int_{\Delta (A)}\widehat{M}(\cdot ,r)d\beta
\end{equation*}%
for $(x,t)\in Q_{T}$ and $(r,\xi )\in \mathbb{R}\times \mathbb{R}^{N}$
arbitrarily fixed. Substituting $u_{1}=\pi (\cdot ,u_{0},Du_{0})$ in (\ref%
{6.8}) and choosing there the particular test functions $\psi
_{0}(x,t,\omega )=\varphi (x,t)\phi (\omega )$ for $\varphi \in \mathcal{C}%
_{0}^{\infty }(Q_{T})$ and $\phi \in B(\bar{\Omega})$ we get by It\^{o}'s
formula, the macroscopic homogenized problem, viz.
\begin{equation}
\left\{
\begin{array}{l}
du_{0}=\left( \Div q(\cdot ,\cdot ,u_{0},Du_{0})-q_{0}(\cdot ,\cdot
,u_{0})\right) dt+\widetilde{M}(u_{0})d\bar{W}\text{\ in }Q_{T} \\
u_{0}=0\text{\ on }\partial Q\times (0,T) \\
u_{0}(x,0)=u^{0}(x)\text{\ in }Q.%
\end{array}%
\right.  \label{6.10}
\end{equation}%
In view of (\ref{6.3}), (\ref{6.10}) admits at least a solution. Moreover
the following uniqueness result holds.

\begin{proposition}
\label{p6.3}Let $u_{0}$ and $u_{0}^{\#}$ be two solutions of \emph{(\ref%
{6.10})} on the same probabilistic system $(\bar{\Omega},\bar{\mathcal{F}},%
\bar{\mathbb{P}}),\bar{W},\bar{\mathcal{F}^{t}}$ with the same initial
condition $u^{0}$. We have that $u_{0}=u_{0}^{\#}$ $\bar{\mathbb{P}}$-almost
surely.
\end{proposition}

\begin{proof}
From the definition of $q_0$ and $\widetilde{M}$, it is not difficult to see
that they are Lipschitz continuous with respect to the variable $u_0$. It
also follows from the definition of the operator $q$ that it satisfies
properties similar to A1.-A3.. Now the proof is quite standard but we give
the detail for sake of completeness. The functions $u_0$ and $u_0^\#$ given
in the proposition satisfy
\begin{equation*}
\begin{split}
dw_0=\left(\Div q(x,t,u_0,Du_0)-\Div q(x,t,u_0^\#,Du_0^\#)\right)dt -\left(
q_0(x,t,u_0)-q_0(x,t,u_0^\#)\right)dt \\
+ (\widetilde{M}(u_0)-\widetilde{M}(u_0^\#))d\bar{W},
\end{split}%
\end{equation*}
where $w_0=u_0-u_0^\#$. For sake of simplicity we will omit the dependence
on the variables $x,t$ in the following computations. Thanks to It\^o's
formula we have
\begin{equation*}
\begin{split}
d|w_0|^2=2\langle\Div (q(u_0,Du_0)-q(u_0^\#,Du_0^\#)), w_0\rangle dt
-2(q_0(u_0)-q_0(u_0^\#), w_0) dt \\
+ |\widetilde{M}(u_0)-\widetilde{M}(u_0^\#)|^2_{L_2}dt+2(\widetilde{M}(u_0)-%
\widetilde{M}(u_0^\#),w_0)d\bar{W}.
\end{split}%
\end{equation*}
Due to the monotonicity of $\Div (q(u,Du))$, the Lipschitz continuity of $%
q_0(.)$ and $\widetilde{M}$ we have that
\begin{equation*}
d|w_0|^2+C |Dw_0|^p\le C |w_0|^2 dt + 2(\widetilde{M}(u_0)-\widetilde{M}%
(u_0^\#),w_0)d\bar{W}.
\end{equation*}
Note that we also used the Cauchy-Schwartz' inequality to get the above
estimate. Integrating over $[0,t]$ and taking the mathematical expectation
to both sides of the latter equations yield
\begin{equation*}
\bar{\mathbb{E}}|w_0(t)|^2\le C \bar{\mathbb{E}}\int_0^t |w_0(s)|^2 ds.
\end{equation*}
Now we can conclude the proof of the proposition by invoking the Gronwall's
lemma.
\end{proof}

\begin{remark}
\label{r6.1}\emph{The pathwise uniqueness result in Proposition \ref{p6.3}
and Yamada-Watanabe's Theorem (see, for instance, \cite{revuz}) implies the
existence of a unique strong probabilistic solution of (\ref{6.10}) on a
prescribed probabilistic system }$(\Omega ,\mathcal{F},\mathbb{P}),\mathcal{F%
}^{t},W$\emph{.}
\end{remark}

We are now in a position to formulate the main homogenization result.

\begin{theorem}
\label{t6.1}Assume that \textbf{A1.-A7.} hold. Suppose moreover that \emph{(%
\ref{5.1})} holds true. Let $2\leq p<\infty $. For each $\varepsilon >0$ let
$u_{\varepsilon }$ be the unique solution of \emph{(\ref{1.1})} on a given
stochastic system $(\Omega ,\mathcal{F},\mathbb{P}),\mathcal{F}^{t},W$
defined as in Section \emph{4}. Then as $\varepsilon \rightarrow 0$, the
whole sequence $u_{\varepsilon }$ converges in probability to $u_{0}$ in $%
L^{2}(Q_{T})$ (i.e., $||u_{\varepsilon }-u_{0}||_{L^{2}(Q_{T})}$ converges
to zero in probability) where $u_{0}$ is the unique strong probabilistic
solution of \emph{(\ref{6.10})}.
\end{theorem}

The main ingredients for the proof of this theorem are the pathwise
uniqueness for (\ref{6.10}) and the following criteria for convergence in
probability whose proof can be found in \cite{GYONGY}.

\begin{lemma}
\label{l6.x} Let $X$ be a Polish space. A sequence of a X-valued random
variables $\{x_{n};n\geq 0\}$ converges in probability if and only if for
every subsequence of joint probability laws, $\{\nu _{n_{k},m_{k}};k\geq 0\}$%
, there exists a further subsequence which converges weakly to a probability
measure $\nu $ such that
\begin{equation*}
\nu \left( \{(x,y)\in X\times X;x=y\}\right) =1.
\end{equation*}
\end{lemma}

Let us set $L^p=L^P(0,T,H)$, $L^{p,2}=L^p(0,T,H)\times L^p(0,T,H) $, $%
\mathfrak{S}^{W}=\mathcal{C}(0,T:\mathcal{U}_0)$, and finally $\mathfrak{S}%
=L^p\times L^p\times \mathfrak{S}^{W}$. For any $S\in \mathcal{B}(L^p)$ we
set $\Pi ^{\varepsilon }(S)=\mathbb{P}(u_{\varepsilon }\in S)$ and $\Pi ^{W}=%
\mathbb{P}(W\in S)$ for any $S\in \mathcal{B}(\mathfrak{S}^{W}) $. Next we
define the joint probability laws :
\begin{align*}
\Pi ^{\varepsilon ,\varepsilon ^{\prime }}& =\Pi ^{\varepsilon }\times \Pi
^{\varepsilon ^{\prime }} \\
\nu ^{\varepsilon ,\varepsilon ^{\prime }}& =\Pi ^{\varepsilon }\times \Pi
^{\varepsilon ^{\prime }}\times \Pi ^{W}.
\end{align*}%
The following tightness property holds.

\begin{lemma}
\label{l6.x2}The collection $\{\nu ^{\varepsilon ,\varepsilon ^{\prime
}};\varepsilon ,\varepsilon ^{\prime }\in E\}$ (and hence any subsequence $%
\{\nu ^{\varepsilon _{j},\varepsilon _{j}^{\prime }}:\varepsilon
_{j},\varepsilon _{j}^{\prime }\in E^{\prime }\}$) is tight on $\mathfrak{S}$%
.
\end{lemma}

\begin{proof}
The proof is very similar to Lemma \ref{l0.4}. For any $\delta >0$ we choose
the sets $Z_\delta$ and $\mathcal{K}_\delta$ exactly as in the proof of
Lemma \ref{l0.4} with appropriate modification on the constants $K_\delta,
L_\delta, M_{\delta }$ so that $\Pi ^{\varepsilon }(Z_{\delta })\geq 1-\frac{%
\delta }{4}$ and $\Pi ^{W}(\mathcal{K}_{\delta })\geq 1-\frac{\delta }{2}$
for every $\varepsilon \in E$. Now let us take $K_{\delta }=Z_{\delta
}\times Z_{\delta }\times \mathcal{K}_\delta$ which is compact in $\mathfrak{%
S}$; it is not difficult to see that $\{\nu ^{\varepsilon ,\varepsilon
^{\prime }}(K_{\delta })\geq (1-\frac{\delta }{4})^{2}(1-\frac{\delta }{2}%
)\geq 1-\delta $ for all $\varepsilon ,\varepsilon ^{\prime }$. This
completes the proof of the lemma.
\end{proof}

\begin{proof}[Proof of Theorem \protect\ref{t6.1}]
To prove Theorem \ref{t6.1} we will mainly use the idea in \cite{Glatt}.
Lemma \ref{l6.x2} implies that there exists a subsequence from $\{\nu
^{\varepsilon _{j},\varepsilon _{j}^{\prime }}\}$ still denoted by $\{\nu
^{\varepsilon _{j},\varepsilon _{j}^{\prime }}\}$ which converges to a
probability measure $\nu $. By Skorokhod's theorem there exists a
probability space $(\bar{\Omega},\bar{\mathcal{F}},\bar{\mathbb{P}})$ on
which a sequence $(u_{\varepsilon _{j}},u_{\varepsilon _{j}^{\prime
}},W^{j}) $ is defined and converges almost surely in $L^{p,2}\times
\mathfrak{S}^{W}$ to a couple of random variables $(u_{0},v_{0},\bar{W})$.
Furthermore, we have
\begin{align*}
Law(u_{\varepsilon _{j}},u_{\varepsilon _{j}^{\prime }},W^{j})& =\nu
^{\varepsilon _{j},\varepsilon _{j}^{\prime }}, \\
Law(u_{0},v_{0},\bar{W})& =\nu .
\end{align*}%
Now let $Z_{j}^{u_{\varepsilon }}=(u_{\varepsilon _{j}},W^{j})$, $%
Z_{j}^{u_{\varepsilon ^{\prime }}}=(u_{\varepsilon _{j}^{\prime }},W^{j})$, $%
Z^{u_{0}}=(u_{0},\bar{W})$ and $Z^{v_{0}}=(v_{0},\bar{W})$. We can infer
from the above argument that $\left( \Pi ^{\varepsilon _{j},\varepsilon
_{j}^{\prime }}\right) $ converges to a measure $\Pi $ such that
\begin{equation*}
\Pi (\cdot )=\bar{\mathbb{P}}((u_{0},v_{0})\in \cdot ).
\end{equation*}%
As above we can show that $Z_{j}^{u_{\varepsilon }}$ and $%
Z_{j}^{u_{\varepsilon ^{\prime }}}$ satisfy (\ref{4.10}) and that $Z^{u}$
and $Z^{v}$ satisfy (\ref{6.10}) on the same stochastic system $(\bar{\Omega}%
,\bar{\mathcal{F}},\bar{\mathbb{P}}),\bar{\mathcal{F}}^{t},\bar{W}$, where $%
\bar{\mathcal{F}^{t}}$ is the filtration generated by the couple $%
(u_{0},v_{0},\bar{W})$. Since we have the uniqueness result above, then we
see that $u^{0}=v^{0}$ almost surely and $u_{0}=v_{0}$ in $L^{p}(0,T;H)$.
Therefore
\begin{equation*}
\Pi \left( \{(x,y)\in L^{p,2};x=y\}\right) =\bar{\mathbb{P}}\left(
u_{0}=v_{0}\text{ in }L^{p}(0,T;H)\right) =1.
\end{equation*}%
This fact together with Lemma \ref{l6.x} imply that the original sequence $%
\left( u_{\varepsilon }\right) $ defined on the original probability space $%
(\Omega ,\mathcal{F},\mathbb{P}),\mathcal{F}^{t},W$ converges in probability
to an element $u_{0}$ in the topology of $L^{p}(0,T;H)$. By a passage to the
limit's argument as in the previous subsection it is not difficult to show
that $u_{0}$ is the unique solution of (\ref{6.10}) (on the original
probability system $(\Omega ,\mathcal{F},\mathbb{P}),\mathcal{F}^{t},W$).
This ends the proof of Theorem \ref{t6.1}.
\end{proof}

\section{A corrector-type result}

Our aim in this section is to prove some general corrector-type results.

Here and henceforth, we set, for a function $\mathbf{v}=(v_{0},v_{1})\in
\mathbb{F}_{0}^{1,p}$, $\mathbb{D}_{y}\mathbf{v}=Dv_{0}+\overline{D}%
_{y}v_{1} $ and $\mathbb{D}\mathbf{v}=Dv_{0}+\partial \widehat{v}_{1}=%
\mathcal{G}_{1}^{N}(\mathbb{D}_{y}\mathbf{v})$. The following result holds.

\begin{theorem}
\label{t7.1}Let the hypotheses be those of Theorem \emph{\ref{t6.1}}. There
exists a continuous nondecreasing function $\overline{\nu }:[0,\infty
)\rightarrow \lbrack 0,\infty )$ with $\overline{\nu }(0)=0$ such that for
all $\Phi =(\psi _{0},\varrho (\psi _{1}))$ with $\psi _{0}\in L^{p}(\Omega
\times (0,T);W_{0}^{1,p}(Q))$ and $\psi _{1}\in L^{p}(\Omega \times
(0,T);W_{0}^{1,p}(Q))\otimes \lbrack A_{\tau }\otimes (A_{y}^{1}/\mathbb{R}%
)] $, if we define $\Phi _{\varepsilon }$ as in \emph{(\ref{5.4})} (see
Corollary \emph{\ref{c5.1}}), then
\begin{equation}
\underset{\varepsilon \rightarrow 0}{\lim \sup }\left\| Du_{\varepsilon
}-D\Phi _{\varepsilon }\right\| _{L^{p}(Q_{T}\times \Omega )^{N}}\leq
\overline{\nu }\left( \left\| \mathbb{D}_{y}\mathbf{u}-\mathbb{D}_{y}\Phi
\right\| _{L^{p}(Q_{T}\times \Omega ;\mathcal{B}_{A}^{p})^{N}}\right) .
\label{7.1}
\end{equation}
\end{theorem}

\begin{proof}
The proof follows closely the one of its homologue in \cite{NA}. We repeat
it here for reader's convenience. Let $F_{0}^{1}$ be the vector space of all
$\Phi $ as in the statement of Theorem \ref{t7.1}. Endowed with an obvious
topology, $F_{0}^{1}$ has $\mathcal{F}_{0}^{\infty }$ as a dense subspace
(this is straightforward). Thus, we first establish (\ref{7.1}) for $\Phi $
in $\mathcal{F}_{0}^{\infty }$. Owing to A2., for $\Phi \in \mathcal{F}%
_{0}^{\infty }$,
\begin{equation*}
\begin{array}{l}
c_{1}\left\Vert Du_{\varepsilon }-D\Phi _{\varepsilon }\right\Vert
_{L^{p}(Q_{T}\times \Omega )^{N}}^{p} \\
\leq \int_{Q_{T}\times \Omega }(a^{\varepsilon }(\cdot ,u_{\varepsilon
},Du_{\varepsilon })-a^{\varepsilon }(\cdot ,u_{\varepsilon },D\Phi
_{\varepsilon })\cdot D(u_{\varepsilon }-\Phi _{\varepsilon })dxdtd\mathbb{P}%
\equiv B_{\varepsilon }.%
\end{array}%
\end{equation*}%
As shown in the proof of Proposition \ref{p6.1}, we see that, as $%
\varepsilon \rightarrow 0$,
\begin{equation*}
B_{\varepsilon }\rightarrow \iint_{Q_{T}\times \Omega \times \Delta
(A)}\left( \widehat{a}(\cdot ,u_{0},\mathbb{D}\mathbf{u})-\widehat{a}(\cdot
,u_{0},\mathbb{D}\Phi )\right) \cdot \mathbb{D}(\mathbf{u}-\Phi )dxdtd%
\mathbb{P}d\beta \equiv B,
\end{equation*}%
where $\mathbf{u}=(u_{0},u_{1})$ is as in Proposition \ref{p6.1}. It follows
that
\begin{equation*}
\underset{\varepsilon \rightarrow 0}{\lim \sup }\left\Vert Du_{\varepsilon
}-D\Phi _{\varepsilon }\right\Vert _{L^{p}(Q_{T}\times \Omega )^{N}}\leq
\left( \frac{B}{c_{1}}\right) ^{\frac{1}{p}}.
\end{equation*}%
But using H\"{o}lder's inequality together with the properties of the
function $a$ (see especially assumption A6. in Section 4), we get
\begin{equation*}
B\leq c_{0}\left\Vert 1+\left\vert u_{0}\right\vert +\left\vert \mathbb{D}%
\mathbf{u}\right\vert +\left\vert \mathbb{D}\Phi \right\vert \right\Vert
_{L^{p}(Q_{T}\times \Omega \times \Delta (A))}^{p-2}\left\Vert \mathbb{D}(%
\mathbf{u}-\Phi )\right\Vert _{L^{p}(Q_{T}\times \Omega \times \Delta
(A))^{N}}^{2},
\end{equation*}%
and by the obvious inequality $\left\vert \mathbb{D}\Phi \right\vert \leq
\left\vert \mathbb{D}\mathbf{u}-\mathbb{D}\Phi \right\vert +\left\vert
\mathbb{D}\mathbf{u}\right\vert $,
\begin{equation*}
\begin{array}{l}
B\leq c_{0}\left( \left\Vert 1+\left\vert u_{0}\right\vert +2\left\vert
Du_{0}\right\vert \right\Vert _{L^{p}(Q_{T}\times \Omega \times \Delta
(A))}+\left\Vert \mathbb{D}(\mathbf{u}-\Phi )\right\Vert _{L^{p}(Q_{T}\times
\Omega \times \Delta (A))^{N}}\right) ^{p-2}\times \\
\;\;\;\;\;\;\;\;\;\times \left\Vert \mathbb{D}(\mathbf{u}-\Phi )\right\Vert
_{L^{p}(Q_{T}\times \Omega \times \Delta (A))^{N}}^{2}.%
\end{array}%
\end{equation*}%
Now, set $\alpha =\left\Vert 1+\left\vert u_{0}\right\vert +2\left\vert
Du_{0}\right\vert \right\Vert _{L^{p}(Q_{T}\times \Omega \times \Delta (A))}$
and
\begin{equation*}
\overline{\nu }(r)=\frac{c_{0}}{c_{1}^{\frac{1}{p}}}r^{\frac{2}{p}}\left(
\alpha +r\right) ^{1-\frac{2}{p}}\text{\ for }r\geq 0\text{.\ \ \ \ \ \ \ \
\ \ \ \ \ \ \ }
\end{equation*}%
Then the function $\overline{\nu }$ is independent of $\Phi $ and satisfies
hypotheses stated in Theorem \ref{t7.1} (this is straightforward by
observing that $\left\Vert \mathbb{D}(\mathbf{u}-\Phi )\right\Vert
_{L^{p}(Q_{T}\times \Omega \times \Delta (A))^{N}}=\left\Vert \mathbb{D}_{y}%
\mathbf{u}-\mathbb{D}_{y}\Phi \right\Vert _{L^{p}(Q_{T}\times \Omega ;%
\mathcal{B}_{A}^{p})^{N}}$). Whence (\ref{7.1}) is shown for $\Phi $ in $%
\mathcal{F}_{0}^{\infty }$.

Now, let $\Phi \in F_{0}^{1}$. Let $(\Psi _{j})_{j}$ be a sequence in $%
\mathcal{F}_{0}^{\infty }$ such that $\Psi _{j}\rightarrow \Phi $ in $%
F_{0}^{1}$ as $j\rightarrow \infty $. Set
\begin{equation*}
\Psi _{j}=(\varphi _{0j},\varrho (\varphi _{1j}))\text{\ and }\Phi =(\psi
_{0},\varrho (\psi _{1})),\;\;\;\;\;\;\;
\end{equation*}%
and define $\Psi _{j,\varepsilon }=\varphi _{0j}+\varepsilon \varphi
_{1j}^{\varepsilon }$\ and $\Phi _{\varepsilon }=\psi _{0}+\varepsilon \psi
_{1}^{\varepsilon }$\ as in (\ref{5.4}). We have
\begin{equation*}
\begin{array}{l}
\underset{\varepsilon \rightarrow 0}{\lim \sup }\left\| Du_{\varepsilon
}-D\Phi _{\varepsilon }\right\| _{L^{p}(Q_{T}\times \Omega )^{N}}\leq
\underset{\varepsilon \rightarrow 0}{\lim \sup }\left\| Du_{\varepsilon
}-D\Psi _{j,\varepsilon }\right\| _{L^{p}(Q_{T}\times \Omega )^{N}}+ \\
\;\;\;\;\;\;\;\underset{\varepsilon \rightarrow 0}{\lim \sup }\left\| D\Psi
_{j,\varepsilon }-D\Phi _{\varepsilon }\right\| _{L^{p}(Q_{T}\times \Omega
)^{N}} \\
\;\;\;\;\;\;\;\;\;\;\leq \overline{\nu }\left( \left\| \mathbb{D}_{y}\mathbf{%
u}-\mathbb{D}_{y}\Psi _{j}\right\| _{L^{p}(Q_{T}\times \Omega ;\mathcal{B}%
_{A}^{p})^{N}}\right) +\underset{\varepsilon \rightarrow 0}{\lim \sup }%
\left\| D\Psi _{j,\varepsilon }-D\Phi _{\varepsilon }\right\|
_{L^{p}(Q_{T}\times \Omega )^{N}}.%
\end{array}%
\end{equation*}%
Now, since $\Psi _{j}\rightarrow \Phi $ in $F_{0}^{1}$, we get $\mathbb{D}%
\Psi _{j}\rightarrow \mathbb{D}\Phi $ in $L^{p}(Q_{T}\times \Omega \times
\Delta (A))^{N}$ as $j\rightarrow \infty $. On the other hand, it can be
easily shown that $\underset{j\rightarrow \infty }{\lim }\underset{%
\varepsilon \rightarrow 0}{\lim }\left\| D\Psi _{j,\varepsilon }-D\Phi
_{\varepsilon }\right\| _{L^{p}(Q_{T}\times \Omega )^{N}}=0$. Hence, taking
the limit (as $j\rightarrow \infty $) of both sides of the last inequality
above, we are led to (\ref{7.1}).
\end{proof}

All the ingredients are now available to state the corrector result.

\begin{corollary}
\label{c7.1}Let the hypotheses be as in Theorem \emph{\ref{t7.1}}. Assume
moreover that
\begin{equation*}
u_{1}\in L^{p}(\Omega \times (0,T);W_{0}^{1,p}(Q))\otimes \lbrack \varrho
(A_{\tau }\otimes (A_{y}^{1}/\mathbb{R}))].
\end{equation*}%
Then, as $\varepsilon \rightarrow 0$,
\begin{equation*}
u_{\varepsilon }-u_{0}-\varepsilon u_{1}^{\varepsilon }\rightarrow 0\text{
in }L^{p}(\Omega \times (0,T);H^{1}(Q)).
\end{equation*}
\end{corollary}

\begin{proof}
It is clear that, on one hand, $\varepsilon u_{1}^{\varepsilon }\rightarrow
0 $\ in $L^{p}(Q_{T}\times \Omega )$ as $\varepsilon \rightarrow 0$; and on
the other hand, due to the tightness property, it can be shown that the
convergence result (\ref{6.1}) still holds with $L^{2}(Q_{T}\times \Omega )$
replaced by $L^{p}(\Omega \times (0,T);L^{2}(Q))$, so that we have $%
u_{\varepsilon }-u_{0}\rightarrow 0$ in $L^{p}(0,T;L^{2}(Q))$ a.s., and
hence $u_{\varepsilon }-u_{0}\rightarrow 0$ in $L^{p}(\Omega \times
(0,T);L^{2}(Q))$. Thus $u_{\varepsilon }-u_{0}-\varepsilon
u_{1}^{\varepsilon }\rightarrow 0$\ in $L^{p}(\Omega \times (0,T);L^{2}(Q))$
as $\varepsilon \rightarrow 0$. It remains to show that $D(u_{\varepsilon
}-u_{0}-\varepsilon u_{1}^{\varepsilon })\rightarrow 0$ in $L^{p}(\Omega
\times (0,T);L^{2}(Q))^{N}$ as $\varepsilon \rightarrow 0$. But, if we set $%
\Phi _{\varepsilon }=u_{0}+\varepsilon u_{1}^{\varepsilon }$, then applying (%
\ref{7.1}), we get
\begin{equation*}
\underset{\varepsilon \rightarrow 0}{\lim \sup }\left\Vert Du_{\varepsilon
}-D\Phi _{\varepsilon }\right\Vert _{L^{p}(Q_{T}\times \Omega
)^{N}}=0\;\;\;\;\;\;\;\;\;\;\;
\end{equation*}%
since $\overline{\nu }\left( \left\Vert \mathbb{D}_{y}\mathbf{u}-\mathbb{D}%
_{y}\Phi \right\Vert _{L^{p}(Q_{T}\times \Omega ;\mathcal{B}%
_{A}^{p})^{N}}\right) =\overline{\nu }(0)=0$, and so $\underset{\varepsilon
\rightarrow 0}{\lim }\left\Vert Du_{\varepsilon }-D\Phi _{\varepsilon
}\right\Vert _{L^{p}(Q_{T}\times \Omega )^{N}}=0$. Thus $D(u_{\varepsilon
}-u_{0}-\varepsilon u_{1}^{\varepsilon })\rightarrow 0$ in $%
L^{p}(Q_{T}\times \Omega )^{N}$, and the result follows from the continuous
embedding $L^{p}(\Omega \times (0,T);L^{2}(Q))\rightarrow L^{p}(Q_{T}\times
\Omega )$. We are therefore done.
\end{proof}

\begin{remark}
\label{r7.1}\emph{If we assume that }$u_{\varepsilon }\rightarrow u_{0}$%
\emph{\ in }$L^{p}(Q_{T}\times \Omega )$\emph{, then the corrector result is
finer and expresses as follows: }%
\begin{equation*}
u_{\varepsilon }-u_{0}-\varepsilon u_{1}^{\varepsilon }\rightarrow 0\text{%
\emph{\ in }}L^{p}(\Omega \times (0,T);W^{1,p}(Q))\text{\emph{\ as }}%
\varepsilon \rightarrow 0\text{.}
\end{equation*}%
\emph{This results holds especially in the deterministic setting since we
have in that case the strong convergence result }$u_{\varepsilon
}\rightarrow u_{0}$\emph{\ in }$L^{p}(Q_{T})$\emph{. In the stochastic
framework, the above results fails in general, and we can not have a better
result than the one in Corollary \ref{c7.1}.}
\end{remark}

\section{Some concrete applications of the abstract homogenization result}

In this section we give some applications of the results of Section 6 to
concrete situations that occurred in some physical setting.

\begin{example}
\label{e8.1}\emph{The homogenization of (\ref{1.1}) can be achieved under
the periodicity assumption: }

\begin{itemize}
\item[(\protect\ref{5.1})$_{1}$] \emph{The functions }$a_{i}(x,t,\cdot
,\cdot ,\mu ,\lambda )$\emph{, }$a_{0}(x,t,\cdot ,\cdot ,\mu )$\emph{\ and }$%
M_{k}(\cdot ,\cdot ,\mu )$\emph{\ are both periodic of period }$1$\emph{\ in
each scalar coordinate, for any fixed }$(x,t)\in \overline{Q}_{T}$\emph{, }$%
(\mu ,\lambda )\in \mathbb{R}\times \mathbb{R}^{N}$\emph{, and }$1\leq i\leq
N$\emph{\ and }$k\geq 1$\emph{.}

\noindent \emph{This leads to (\ref{5.1}) with }$A=\mathcal{C}_{\text{\emph{%
per}}}(Y\times Z)=\mathcal{C}_{\text{\emph{per}}}(Y)\odot \mathcal{C}_{\text{%
\emph{per}}}(Z)$\emph{\ (the product algebra, with }$Y=(0,1)^{N}$\emph{\ and
}$Z=(0,1)$\emph{), and hence }$B_{A}^{r}=L_{\text{\emph{per}}}^{r}(Y\times
Z) $\emph{\ for }$1\leq r\leq \infty $\emph{.}
\end{itemize}
\end{example}

\begin{example}
\label{e8.2}\emph{The above functions in (\ref{5.1})}$_{1}$\emph{\ are both
Besicovitch almost periodic in }$(y,\tau )$\emph{. This amounts to (\ref{5.1}%
) with }$A=AP(\mathbb{R}_{y,\tau }^{N+1})=AP(\mathbb{R}_{y}^{N})\odot AP(%
\mathbb{R}_{\tau })$\emph{\ (}$AP(\mathbb{R}_{y}^{N})$\emph{\ the Bohr
almost periodic functions on }$\mathbb{R}_{y}^{N}$\emph{).}
\end{example}

\begin{example}
\label{e8.3}\emph{The homogenization problem for (\ref{1.1}) can be
considered under the assumption }

\begin{itemize}
\item[(\protect\ref{5.1})$_{2}$] $a_{i}(x,t,\cdot ,\cdot ,\mu ,\lambda )$%
\emph{\ is weakly almost periodic while the functions }$a_{0}(x,t,\cdot
,\cdot ,\mu )$\emph{\ and }$M_{k}(\cdot ,\cdot ,\mu )$\emph{\ are almost
periodic in the Besicovitch sense. This yields (\ref{5.1}) with }$A=WAP(%
\mathbb{R}_{y}^{N})\odot WAP(\mathbb{R}_{\tau })$\emph{\ (}$WAP(\mathbb{R}%
_{y}^{N})$\emph{, the algebra of continuous weakly almost periodic functions
on }$\mathbb{R}_{y}^{N}$\emph{; see e.g., \cite{17}).}
\end{itemize}
\end{example}

\begin{example}[Homogenization in the Fourier-Stieltjes algebra]
\label{e8.4}\emph{Let us first and foremost define the Fourier-Stieltjes
algebra on }$\mathbb{R}^{N}$. \emph{The Fourier-Stieltjes algebra on }$%
\mathbb{R}^{N}$\emph{\ is defined as the closure in }$BUC(\mathbb{R}^{N})$%
\emph{\ (the bounded uniformly continuous functions on }$\mathbb{R}^{N}$%
\emph{) of the space }
\begin{equation*}
FS_{\ast }(\mathbb{R}^{N})=\left\{ f:\mathbb{R}^{N}\rightarrow \mathbb{R}%
,\;f(x)=\int_{\mathbb{R}^{N}}\exp (ix\cdot y)d\nu (y)\text{\ \emph{for some}
}\nu \in \mathcal{M}_{\ast }(\mathbb{R}^{N})\right\}
\end{equation*}%
\emph{where }$\mathcal{M}_{\ast }(\mathbb{R}^{N})$\emph{\ denotes the space
of complex valued measures }$\nu $\emph{\ with finite total variation: }$%
\left\vert \nu \right\vert (\mathbb{R}^{N})<\infty $\emph{. We denote it by }%
$FS(\mathbb{R}^{N})$\emph{.}

\emph{Since by \cite{17} any function in }$FS_{\ast }(\mathbb{R}^{N})$\emph{%
\ is a weakly almost periodic continuous function, we have that }$FS(\mathbb{%
R}^{N})\subset WAP(\mathbb{R}^{N})$\emph{. It is a well known fact that }$FS(%
\mathbb{R}^{N})$\emph{\ is an ergodic algebra which is translation invariant
(this follows from the fact that }$FS_{\ast }(\mathbb{R}^{N})$\emph{\ is
translation invariant), so that all the hypotheses of Theorem \ref{t3.2} are
satisfied with any algebra }$A=FS(\mathbb{R}^{N})\odot A_{\tau }$\emph{, }$%
A_{\tau }$\emph{\ being any algebra wmv on }$\mathbb{R}_{\tau }$\emph{.}

\emph{This being so, we aim at solve homogenization problem for (\ref{1.1})
under the assumption}

\begin{itemize}
\item[(\protect\ref{5.1})$_{3}$] $a_{i}(x,t,\cdot ,\cdot ,\mu ,\lambda )\in
B_{A_{\tau }}^{p^{\prime }}(\mathbb{R}_{\tau };B_{FS}^{p^{\prime }}(\mathbb{R%
}_{y}^{N}))$\emph{, }$a_{0}(x,t,\cdot ,\cdot ,\mu )$\emph{, }$M_{k}(\cdot
,\cdot ,\mu )\in B_{A_{\tau }}^{2}(\mathbb{R}_{\tau };B_{FS}^{2}(\mathbb{R}%
_{y}^{N}))$\emph{\ for any }$(\mu ,\lambda )\in \mathbb{R\times R}^{N}$\emph{%
, and for all }$(x,t)\in \overline{Q}_{T}$\emph{, }$(1\leq i\leq N)$

\noindent \emph{where }$B_{FS}^{p^{\prime }}(\mathbb{R}_{y}^{N})$\emph{\
denotes the closure of the algebra }$FS(\mathbb{R}_{y}^{N})$\emph{\ with
respect to the seminorm }$\left\Vert \cdot \right\Vert _{p^{\prime }}$\emph{%
, and }$A_{\tau }$\emph{\ is any arbitrary algebra wmv on }$\mathbb{R}_{\tau
}$\emph{. We are then led to (\ref{5.1}) with }$A=FS(\mathbb{R}^{N})\odot
A_{\tau }$\emph{.}
\end{itemize}
\end{example}

\begin{remark}
\label{r8.1}\emph{It should be stressed that the problems solved in Examples %
\ref{e8.3} and \ref{e8.4} are new in the literature as far as the
homogenization of SPDEs is concerned.}
\end{remark}

\begin{acknowledgement}
\emph{The authors would like to thank the anonymous referee for valuable
remarks and suggestions.}
\end{acknowledgement}

\end{document}